\documentclass[reqno]{amsart}

\usepackage{amsmath,amssymb,cite}
\usepackage{graphics,epsfig,color}
\usepackage[mathscr]{euscript}
\usepackage{mathtools}
\usepackage{comment}
\mathtoolsset{showonlyrefs}
\usepackage{xcolor}
\usepackage{hyperref}
\usepackage[c]{esvect}
\usepackage{cancel}

\usepackage[normalem]{ulem}

\definecolor{bblue}{rgb}{.2,0.2,.8}
\usepackage{graphicx} 
\usepackage{tikz}
\usepackage{subfigure}
\usetikzlibrary{graphs, shapes.geometric, arrows.meta}
\usetikzlibrary{decorations.pathmorphing}
\usepackage{circuitikz}
\usetikzlibrary{trees, positioning}
\usetikzlibrary{fit, shapes.geometric}
\usepackage[all]{xy}

\theoremstyle{plain}
\newtheorem{theorem}{Theorem}[section]
\newtheorem{proposition}[theorem]{Proposition}
\newtheorem{lemma}[theorem]{Lemma}
\newtheorem{corollary}[theorem]{Corollary}

\theoremstyle{definition}
\newtheorem{definition}[theorem]{Definition}
\newtheorem{assumption}[theorem]{Assumption}

\theoremstyle{remark}
\newtheorem{remark}[theorem]{Remark}

\newcommand{\R}{\mathbb R}
\newcommand{\T}{\mathbb T}
\newcommand{\N}{\mathbb N}

\renewcommand{\d}{\mathrm d}

\numberwithin{equation}{section}
\numberwithin{theorem}{section}

\def\\{\par\medskip}

\let\0=\noindent

\newcommand{\id}{{1 \mskip -5mu {\rm I}}}
\renewcommand{\epsilon}{\varepsilon}
\renewcommand{\hat}{\widehat}

\begin{document}

\title[Invariant measure of multiscale chains]{The Invariant Measure of Multiscale Markov Chains via Fast Arborescence Factorization}

\author{Diego Alberici}
\address{\noindent Diego Alberici \hfill\break\indent 
	DISIM, Universit\`a dell'Aquila
	\hfill\break\indent 
	67100 Coppito, L'Aquila, Italy
}
\email{diego.alberici@univaq.it}

\author{Davide Gabrielli}
\address{\noindent Davide Gabrielli \hfill\break\indent 
 DISIM, Universit\`a dell'Aquila
\hfill\break\indent 
67100 Coppito, L'Aquila, Italy
}
\email{davide.gabrielli@univaq.it}

\author{Giulia Pallotta}
\address{\noindent Giulia Pallotta \hfill\break\indent 
	DISIM, Universit\`a dell'Aquila
	\hfill\break\indent 
	67100 Coppito, L'Aquila, Italy
}
\email{giulia.pallotta@graduate.univaq.it}

\begin{abstract}
	We consider a family of continuous-time Markov chains with finite strongly connected transition graph and rates $\left(r_N\right)_{N>0}$ depending on a parameter $N$, so that, when $N$ is large, transitions may happen on different time scales. Under suitable general assumptions on the asymptotic behavior of the rates, we give a recursive characterization of the limiting invariant measure. The recursion is encoded in a forest structure equivalent to the one recently developed in the analysis of dynamical aspects of metastability \cite{BL,LX}.

Our proof is based on a combinatorial representation of the invariant measure, given by the Markov chain tree theorem. Basic steps are the reduction of the chain by a trace process, the introduction of an effective dynamics, and a careful analysis of the set of relevant arborescences in the expansion. In particular we use a factorization of fast arborescences. As a byproduct we obtain properties of the arborescences of generalized star-delta reductions of weighted digraphs.
\end{abstract}

\maketitle
\thispagestyle{empty}

\section*{Introduction}

An interesting problem in Probability Theory is to understand the behavior of a Markovian dynamics whose components act on different time scales.
This framework is commonly encountered in applications and is also highly compelling from a theoretical point of view.
This serves as the core motivation for the theory of metastability, which - originating from the seminal papers \cite{Eyr35, Kra40} and subsequently \cite{CGOV, FW70, OS95, BEGK} - has developed into a broad and prominent general theory utilizing various approaches and techniques; see for example \cite{BH, OV05, L1} for general references.

In particular, in the context of spin glasses degrees of freedom evolving on different time scales were introduced in \cite{H,PCS,DFM}. 
In fact, many time scales are expected to naturally emerge in the mean-field spin glass dynamics, due to a free-energy landscape characterized by nested traps (see for example \cite{BCKM,Pa,bAC} for introductions to the topic, and \cite{bAJ} 
and references therein).

\smallskip

We consider here a simple, specific model corresponding to a continuous-time Markov chain with a fixed, finite state space and general transition rates depending on a parameter $N$ and operating on different time scales when $N$ becomes large.
Moreover, we focus specifically on the asymptotic behavior of the invariant measure.
This is a classic problem to which several papers have been dedicated, see for example \cite{BlR,SA, Mey,Y} and references therein. The fundamental mechanism for understanding the limiting behavior of the invariant measure relies on the classic averaging principle, according to which the effective slow dynamics is averaged over the invariant measure of the fast component. While a direct application of this principle provides the answer in some simple cases, it is generally not sufficient in more complex settings, as identifying fast and slow components is not always straightforward and multiple time scales may coexist. In particular, the principle must be applied recursively a finite number of times, suitably identifying the fast components at each iteration step. 
We show that these iterative steps can be codified in a directed forest graph, revealing a hierarchical structure of the invariant measure.
Starting from the deepest root and proceeding via products along paths, we obtain the value of the limiting measure at the corresponding leaves.
This forest is identical to the one obtained by analyzing the dynamic metastable behavior through the trace process in \cite{BL,LX}.
\smallskip

Our approach is combinatorial and is based on the representation of the invariant measure via the Markov chain tree theorem.
The limiting invariant measure is related to the weights of the dominating arborescences.
We show that these dominating arborescences can be obtained by a factorization in terms of fast forests of the \textit{reduced dynamics} and arborescences of an \textit{effective dynamics}.
This decomposition is repeated at every step of iteration, by taking the effective dynamics as the new chain.

The reduced dynamics is the trace process on the rapidly recurrent states. We interpret it in terms of a network reduction, that we analyze in detail. Since we do not assume reversibility of the dynamics, this corresponds to a generalized star-delta reduction (see \cite{BF,GH}).
After that the effective dynamics is an application of the averaging principle.

\smallskip

Our main result is Theorem \ref{ilteo}, which expresses the limiting invariant measure as a convex combination of the invariant measures of the fast dynamics within the closed irreducible classes. The coefficients of this combination are given by the invariant measure of the effective dynamics.
By recursively applying the procedure, our second main result, Theorem \ref{ilteo2}, expresses the limiting invariant measure as a finite product of weights along the paths from the root to the leaves of a suitable forest. This forest encodes the structure of the metastable behaviors \cite{BL,LX}.

\smallskip

A future problem of interest is the study of the full expansion of the invariant measure in terms of arborescences at all orders and not only the leading one, like in \cite{conL,L, dGM}. A further problem is to investigate state spaces that are infinite or whose cardinality increases together with the parameter $N$.

\smallskip
The paper is organized as follows.

In Section \ref{sec: notation}, we introduce the model, discuss the assumptions, and illustrate the basic combinatorial constructions.

In Section \ref{sec: main}, we introduce the reduced and effective dynamics, state our main results, and prove some basic propositions.

In Section \ref{sec: iteration}, we outline the recursive structure derived from the main result.

In Section \ref{sec:stardelta}, we illustrate a generalized star-delta reduction for directed weighted graphs.

In Section \ref{sec:proof}, we present the core of the proof of the main theorem.

In Section \ref{sec:examples}, we discuss several examples, including a multi-scale product space and a boundary driven exclusion process.
\smallskip

\section{Multiscale Markov chain and combinatorial constructions} \label{sec: notation}

\subsection{Multiscale Markov chain}
Let $G=(V,E)$ be a strongly connected directed graph with finite vertex set $V$ and edge set $E$. To every edge $(x,y) \in E$ we associate a positive weight $r_N(x,y)>0$ depending on a parameter $N >0$.
It is sometimes convenient to set $r_N(x,y)=0$ for every pair of vertices $(x,y)\notin E$.
We call $X^N_t$ a continuous time Markov chain on $(V,E)$ with transition rates $r_N$. Since for finite $N$ the chain is irreducible, it admits a unique (strictly positive) invariant measure $\pi_N$ that is characterized by the stationarity condition
\begin{equation}\label{stazcond}
	\pi_N(x)\sum_{y:(x,y)\in E}r_N(x,y)\,=\sum_{y: (y,x)\in E}\pi_N(y)\;r_N(y,x)\,,\quad x\in V\,.
\end{equation}
Our aim is to determine the limiting invariant measure 
\begin{equation}
\pi:=\lim_{N \rightarrow{\infty}}\pi_N
\end{equation}
under suitable assumptions on the rates and our approach is based on the combinatorial representation of $\pi_N$ given by the Markov chain tree theorem.
The fact that we are interested in Markov chains with \emph{multiscale behavior} is formalized in the following basic property, to which additional hypotheses will be added when necessary.
\begin{assumption}\label{assumption1}
	We assume that for each pair of edges $(x,y),(w,z)\in E$ 
	the limit
	\begin{equation}\label{limc}
		\lim_{N\to\infty}\frac{r_N(x,y)}{r_N(w,z)}\,
	\end{equation}
	exists and is either finite and positive, zero or $\infty$.
\end{assumption}
Assumption \ref{assumption1} defines a partition of $E$ into equivalence classes $E=\cup_{i=1}^F E_i$ with respect to the \textit{scale} or \textit{velocity} of the rates: the edges $(x,y)$ and $(w,z)$ belong to the same class if the limit \eqref{limc} is finite and strictly positive.
There is a natural total order relation on these velocity classes: we may assume that for $j>i$ the edges in $E_j$ are faster than those in $E_i$, 
that is the limit \eqref{limc} is $\infty$ for $(x,y)\in E_j$ and $(w,z)\in E_i$.
$E_F$ denotes the maximum equivalence class in this total order: it contains the edges that are faster than any other and we call them the \emph{fast edges}.

\subsection{Markov chain tree theorem}

Let us start with the definition of \emph{arborescence}, a generalization of spanning tree in the directed context.

\begin{definition}
Let $\tau$ be a directed subgraph of $(V,E)$ and let $x\in V$. $\tau$ is an \emph{arborescence rooted at $x$} (or \emph{directed towards $x$}) if for every $y \in V$, $\tau$ contains a unique directed path going from $y$ to $x$.	

We denote by $\mathcal{T}_x$ the set of arborescences of $(V,E)$ rooted at $x$ and we set $\mathcal{T}=\bigcup_{x \in V}\mathcal{T}_x$.
\end{definition}
In the previous definition the singleton $\{x\}$ has to be considered as a path from $x$ to itself.
An arborescence $\tau\in\mathcal T_x$ has the following properties:
\begin{itemize}
        \item $\tau$ contains no edges going out from the root $x$;
		\item for every $y \in V \setminus \{x\}$, $\tau$ contains a unique edge going out from $y$;
		\item $\tau$ contains no cycles;
		\item $\tau$ is spanning on $V$.
\end{itemize}
\smallskip

Every arborescence $\tau\in\mathcal T$ is given a weight defined as the product of the weights of its constituent edges:
\begin{equation}
	r_N(\tau) := \prod_{(x,y) \in \tau} r_N(x,y) \;.
\end{equation}
For a generic subgraph $\gamma$ of $(V,E)$ we set
$r_N(\gamma):=\prod_{(x,y)\in \gamma}r_N(x,y)$. Moreover, for a collection $\mathcal G$ of subgraphs of $(V,E)$, we set
$r_N(\mathcal G):=\sum_{\gamma\in \mathcal G}r_N(\gamma)$.
In particular
\begin{equation}
	r_N(\mathcal{T}_x) = \sum_{\tau \in \mathcal{T}_x} \prod_{(x,y) \in \tau} r_N(x,y)
\end{equation}
and similarly $r_N(\mathcal{T})=\sum_{x \in V}r_N(\mathcal{T}_x)$.
\begin{figure}
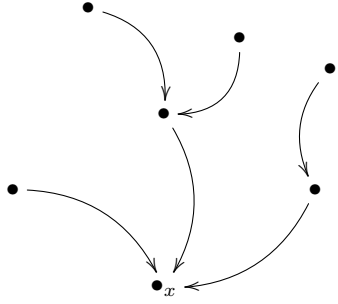

	\centering
	\mbox{ \xygraph{
			!{<0cm,0cm>;<1cm,0cm>:<0cm,1cm>::}
			!{(0,0) }*+{\bullet}="a"
			!{(2,1) }*+{\bullet}="b"
			!{(4,0) }*+{\bullet}="c"
			!{(2,-1.3) }*+{\bullet_{x}}="d"
			!{(1,2.4) }*+{\bullet}="e"
			!{(3,2) }*+{\bullet}="f"
			!{(4.2,1.6) }*+{\bullet}="g"
			"e":@/^0.4cm/"b"
			"b":@/^0.4cm/"d"
			"a":@/^0.4cm/"d"
			"f":@/^0.4cm/"b"
			"g":@/_0.3cm/"c"
			"c":@/^0.4cm/"d"
		}
	}
	\caption{An arborescence rooted at $x$}
\end{figure}
These combinatorial objects allow us to provide an expression for the invariant measure of Markov chains, under uniqueness condition. The following is a classical result
\begin{theorem}[Markov chain tree theorem]
	\label{matrixtreeth}
	Let $X^N_t$ be a continuous-time Markov chain with transition graph $(V,E)$ and positive transition rates $r_N$.
	If $(V,E)$ is strongly connected, then $X^N_t$ admits a unique invariant measure $\pi_N$, which is strictly positive and given by
	\begin{equation}
		\label{markovchaingeneral}
		\pi_N(x)=\frac{r_N(\mathcal{T}_x)} {r_N(\mathcal{T})}\,,\quad x\in V\,.
	\end{equation}
\end{theorem}
For a proof we refer for example to \cite{AT}. We point out also a continuous version in \cite{ACG} and an application to discrete Hamilton-Jacobi equations and large deviations in \cite{AGP}. The results of this paper could possibly have applications in such frameworks.

\subsection{Equivalence classes, DAG and invariant measures}
We briefly recall the theory of Markov chain \textit{equivalence classes} from a graph-theoretic perspective.
In this subsection $(\mathfrak V,\mathfrak E)$ is a generic digraph, not necessarily strongly connected, and $R_N$ is a positive weight associated to each edge in $\mathfrak E$.
The construction introduced here will be applied in the following to several digraphs. First of all the reader may think  to the fast transition graph $(V,E_F)$.
\smallskip

Two nodes $x,y\in\mathfrak V$ belong to the same equivalence class $Q\subseteq \mathfrak V$ if there exist directed paths in $(\mathfrak V,\mathfrak E)$ both from $x$ to $y$ and from $y$ to $x$. In the theory of Markov chains such equivalence classes are called \emph{communicating classes} or \emph{irreducible classes}.
We denote $\mathfrak E[Q]\subseteq \mathfrak E$ the set of edges with both extrema in $Q$ and observe that the subgraph $(Q,\mathfrak E[Q])$ is strongly connected.
\smallskip

Now, let $\mathcal Q$ be the collection of all equivalence classes and let $(\mathcal Q, \mathcal E)$ be the digraph defined by: $(Q,Q')\in \mathcal E$ iff there exist $x\in Q$, $y\in Q'$ such that $(x,y)\in \mathfrak E$.
Observe that by definition  $(\mathcal Q, \mathcal E)$ is a \emph{directed acyclic graph} (DAG), i.e. it contains no directed cycles.
\begin{figure}[h]
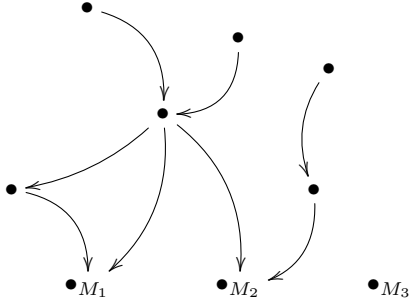

	\centering
	\mbox{ \xygraph{
			!{<0cm,0cm>;<1cm,0cm>:<0cm,1cm>::}
			!{(0,0) }*+{\bullet}="a"
			!{(2,1) }*+{\bullet}="b"
			!{(4,0) }*+{\bullet}="c"
			!{(1,-1.3) }*+{\bullet_{M_1}}="d" 
			!{(3,-1.3) }*+{\bullet_{M_2}}="k"
			!{(5,-1.3) }*+{\bullet_{M_3}}="z"
			!{(1,2.4) }*+{\bullet}="e"
			!{(3,2) }*+{\bullet}="f"
			!{(4.2,1.6) }*+{\bullet}="g"
			"e":@/^0.4cm/"b"
			"b":@/^0.4cm/"d"
			"a":@/^0.4cm/"d"
			"b":@/^0.2cm/"a"
			"f":@/^0.4cm/"b"
			"b":@/^0.4cm/"k"
			"g":@/_0.25cm/"c"
			"c":@/^0.4cm/"k"
		}
	}
	\caption{A directed acyclic graph (DAG) with local minima $M_1,M_2,M_3$.}
\end{figure}
\smallskip

A DAG naturally induces a partial order on $\mathcal Q$: we may say that $Q\geq Q'$ if there exists a directed path from $Q$ to $Q'$ in $(\mathcal Q,\mathcal E)$.
Since the DAG is finite, it has at least one \textit{local minimum}, i.e. there exists $M\in \mathcal Q$ such that the condition $Q'\leq M$ implies $Q'=M$.
Let $\mathcal M\subseteq \mathcal Q$ denote the collection of classes that are local minima. In the theory of Markov chains they are called \emph{closed classes}, since no edge $e\in\mathfrak E$ can exit from them.
\smallskip

Recall that the equivalence classes $Q ,M$ are subsets of $\mathfrak V$ while $\mathcal Q, \mathcal M$ are collections of disjoint subsets of $\mathfrak V$. 
By the classical theory of Markov chains, the set of invariant measures of a continuous-time Markov chain with weighted transition graph $(\mathfrak V,\mathfrak E, R_N)$ is the $(|\mathcal M|-1)$-dimensional simplex
\begin{equation}\label{invclass}
	\left\{\sum_{M\in \mathcal M}c_M\;\mu_M^N\;\;\bigg|\;\; 0\leq c_M\leq1\;\forall\,M\in\mathcal M\;, \sum_{M\in\mathcal M} c_M=1\right\} \;,
\end{equation}
where each $\mu_M^N$ denotes the unique invariant measure of the irreducible Markov chain with strongly connected graph $(M, \mathfrak E[M], R_N)$.
Note that $\mu_M^N$ can be naturally interpreted as a positive probability measure on $M$ or as a (non-negative) probability measure on $\mathfrak V$.
In particular, every invariant measure in the set \eqref{invclass} is supported on $\textrm{Su}\left(\mathcal M\right):=\bigcup_{M\in \mathcal M}M\subseteq \mathfrak V$. 
\smallskip

Let us now prove a simple proposition for the case $|\mathcal M|=1$ generalizing Theorem \ref{matrixtreeth}. We will need it in the following and it does not seem to appear in the literature.

\begin{proposition} Let $(\mathfrak V,\mathfrak E)$ be a digraph with positive weights $R_N$ and let $(\mathcal Q, \mathcal E)$ be the corresponding DAG of irreducible classes defined above. Then:

$(\mathcal Q, \mathcal E)$ has a unique minimal class (i.e. $|\mathcal M|=1$) if and only if $(\mathfrak V,\mathfrak E)$ admits a rooted arborescence.

In this case, the Markov chain with transition graph $(\mathfrak V,\mathfrak E, R_N)$ has a unique invariant measure, given by \eqref{markovchaingeneral}. 
	\label{Unicomassimo}
\end{proposition}

\begin{proof}
Suppose $|\mathcal M|=1$. By \eqref{invclass} the invariant measure is unique. The acyclic graph $(\mathcal Q, \mathcal E)$ has an unique local minimum $M\in\mathcal M$, so there exist at least one edge of $\mathcal E$ exiting from each $Q\in\mathcal Q,\,Q\neq M$. We choose arbitrarily one edge $(Q,Q')$ exiting from each $Q\neq M$ and, since the graph is acyclic, we obtain an arborescence of $(\mathcal Q, \mathcal E)$ rooted at $M$.
Moreover for each edge of this arborescence we can choose an edge $e_Q=(x,x')\in \mathfrak E$ with $x\in Q$, $x'\in Q'$,
and then we choose an arborescence $\tau_Q$ rooted at $x$ of the strongly connected graph $(Q,\mathfrak E[Q])$.
Finally, we choose an arborescence $\tau_M$ of the strongly connected graph $(M,\mathfrak E[M])$.
Considering altogether the edges of $\tau_M$, $\{\tau_Q\}_{Q\neq M}$ and $\{e_Q\}_{Q\neq M}$, we obtain an arborescence on $(\mathfrak V,\mathfrak E)$.
\smallskip
	
Conversely, if there exists at least one arborescence of $(\mathfrak V,\mathfrak E)$, say rooted at $x$, then the only minimal class in $(\mathcal Q, \mathcal E)$ can be the class containing $x$, since from any other vertex of $\mathfrak V$ there exists a directed path leading to $x$.
\smallskip
		
Now, assume the set of arborescences is not empty. The denominator in \eqref{markovchaingeneral} is non zero and the formula is well defined.
Inserting \eqref{markovchaingeneral} into \eqref{stazcond} (written with rates $R_N$) and simplifying the denominator, we obtain on both sides the same weight $R_N(\mathcal U_x)$.  $\mathcal U_x$ is the set of spanning subgraphs of $(\mathfrak V,\mathfrak E)$ where there is exactly one edge going out from every vertex and $x$ belongs to the unique cycle.
%
\end{proof}


\subsection{Existence of the limiting invariant measure} We require further assumptions on the graphs $(V,E,r_N)$ (inspired by those introduced in \cite{BL, LX} and later adopted in \cite{FK}): 

\begin{assumption}\label{FRL}
 The rates $
 r_N$ are such that for every pair of collections of edges $E',E''\subseteq E$ with $|E'|=|E''|$, the limit
\begin{equation}
    \lim_{N \to \infty} \frac{\prod_{e\in E'} r_N(e)}{\prod_{e\in E''} r_N(e)}
\end{equation}
exists and is either finite and positive, zero or $\infty$.
\end{assumption}

\begin{assumption}\label{FRL2}
The rates $r_N$ are such that for every pair of collections of non-negative integer numbers $(m_e\geq 0)_{e\in E}$ and $(n_e\geq 0)_{e\in E}$ such that  $\sum_{e\in E}m_e=\sum_{e\in E}n_e$, the limit
\begin{equation}
		\lim_{N \to \infty} \frac{\prod_{e\in E}r_N(e)^{m_e}}{\prod_{e\in E}r_N(e)^{n_e}}
\end{equation}
exists and is either finite and positive, zero or $\infty$.
\end{assumption}
Assumption \ref{FRL2} is stronger than Assumptions \ref{FRL} and \ref{assumption1}. Rates $r_N$ with polynomial or exponential growth in $N$ satisfy such hypotheses.

\smallskip

Since every arborescence has the same number of edges $|V|-1$, Assumption \ref{FRL} guarantees that for every pair $\tau, \tau'\in \mathcal T$ the limit $\lim_{N\to \infty}\frac{r_N(\tau)}{r_N(\tau')}$ exists.
\smallskip

This means that, as done for edges after Assumption \ref{assumption1}, we can partition the set $\mathcal T$ of arborescences into equivalence classes with a total order relation with respect to the \textit{scale} or \textit{velocity} of their weights.
We denote by $\mathcal T^F$ the \textit{fast} equivalence class of arborescences: $\tau\in\mathcal T^F$ if and only if $\tau\in\mathcal T$ and for every $\tau'\in\mathcal T$ the limit $\lim_{N\to \infty}\frac{r_N(\tau')}{r_N(\tau)}$ is finite.

Similarly, $\mathcal T^F_x$ denotes the fastest equivalence class of arborescences rooted at $x$. Notice that $\mathcal T^F_x$ is not empty but may not contain arborescences of $\mathcal T^F$: it is possible that the fastest arborescences rooted at $x$ are slower than the fast arborescences with arbitrary root.

Identifying the fast arborescences and the corresponding scale is not an easy task and it is essentially one of the aims of this paper. We stress that fast arborescences are not necessarily those consisting solely of fast edges (this occurs only when such arborescences exist), but those belonging to the fastest velocity class of arborescences.

\begin{proposition}[Existence of the limiting invariant measure]
\label{existence}
Let $(V,E,r_N)$ be a strongly connected weighted graph satisfying Assumption \ref{FRL}.
Let $\pi_N$ be the unique invariant measure.
Then $\pi=\lim_{N \rightarrow \infty} \pi_N$ exists and is a probability measure, moreover
\begin{equation} \label{limF}
	 \pi(x) \,=\, 
	 \lim_{N\to \infty}\frac{r_N\left(\mathcal T^F_x\right)}{r_N\left(\mathcal T^F\right)} \,,\quad x\in V\,.
\end{equation}
This is non-zero when $\mathcal T_x^F\cap\mathcal T^F$ is not empty.
\end{proposition}

\begin{proof}
The fact that the limit $\pi$ -if exists- is a probability measure is guaranteed since our graph is finite. 
We rewrite the Markov chain tree theorem \eqref{markovchaingeneral} as
\begin{equation}\label{manga}
\pi_N(x) \,=\,
\frac{r_N\left(\mathcal T_x\right)}{r_N\left(\mathcal T\right)} \,=\,
\frac{r_N\left(\mathcal T^F_x\right)}{r_N\left(\mathcal T^F\right)}\,\; \frac{1+\frac{r_N\left(\mathcal T_x \setminus \mathcal T_x^F\right)}{r_N\left(\mathcal T^F_x\right)}}{1+\frac{r_N\left(\mathcal T \setminus \mathcal T^F\right)}{r_N\left(\mathcal T^F\right)}}\;.
\end{equation}

Let $\mathcal G,\,\mathcal G^F \subseteq \mathcal T$ be two collections of arborescences such that $\mathcal G^F\cap \mathcal T^F$ is not empty. Let $g\in \mathcal G$ be one among the fastest arborescences in $\mathcal G$ and let $g^*\in \mathcal G^F\cap \mathcal T^F$ be a fast arborescence.
We have
\begin{equation}\label{toy}
\frac{r_N(\mathcal G)}{r_N(\mathcal G^F)}\,=\, \frac{r_N\left(g\right)}{r_N\left(g^*\right)}\,\; \frac{1+\sum_{h\in \mathcal G,\, h\neq g}\frac{r_N(h)}{r_N(g)}}{1+\sum_{h^*\in \mathcal G^F,\, h^*\neq g^*}\frac{r_N(h^*)}{r_N(g^*)}}\;.
\end{equation} 
We use Assumption \ref{FRL}. When $N\to\infty$ the first ratio on the right hand side converges to a finite positive value if $g$ is a fast arborescence, i.e. $g\in\mathcal T^F$, or to zero otherwise. The second ratio on the right hand side converges to a finite positive value. Therefore expression \eqref{toy} converges to a finite value, possibly zero.
	 	
From this we deduce that all terms $\frac{r_N\left(\mathcal T^F_x\right)}{r_N\left(\mathcal T^F\right)}$, $\frac{r_N\left(\mathcal T_x \setminus \mathcal T_x^F\right)}{r_N\left(\mathcal T^F_x\right)}$, $\frac{r_N\left(\mathcal T \setminus \mathcal T^F\right)}{r_N\left(\mathcal T^F\right)}$ in \eqref{manga} converge and the limit of the first one may be either positive or zero, while the other two vanish.
Therefore we conclude \eqref{limF}.
\end{proof}

The argument above shows that only fast arborescences contribute to the limit of $\pi_N$ as $N\to\infty$. Consequently, if one shows that a certain class of arborescences is not fast, then it can be neglected from the computation.
\smallskip

\section{Main result}
\label{sec: main}

Let $(V,E,r_N)$ be our strongly connected weighted graph. We call $(V,E_F)$ the \emph{fast subgraph} of our original graph $(V,E)$. Recall that $E_F$ is the class of fast edges of $E$.
 $(V,E_F)$ is not necessarily strongly connected and we can construct the DAG of equivalence classes, as illustrated in Section \ref{sec: notation} (see Figure \ref{primografo}). We denote this DAG by $\left(\mathcal Q, \mathcal E\right)$ and by $Q\in \mathcal Q$ a generic equivalence class: $Q$ is an equivalence class with respect to communication using fast edges only.
We denote by $\mathcal M\subseteq \mathcal Q$ the set of local minima of the partial order induced by the DAG and by $M\in \mathcal M$ a generic element.
\begin{figure}
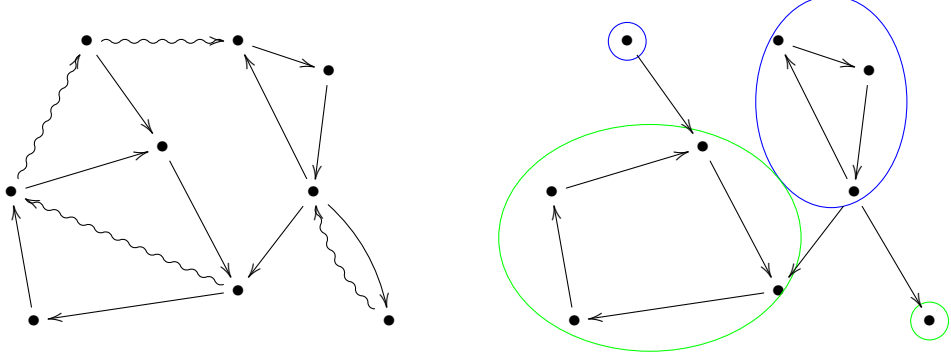

	\centering
	\mbox{ \xygraph{
			!{<0cm,0cm>;<1cm,0cm>:<0cm,1cm>::}
	!{(0,0) }*+{\bullet}="a"
			!{(2,0.6) }*+{\bullet}="b"
			!{(4,0) }*+{\bullet}="c"
			!{(3,-1.3) }*+{\bullet}="d"
			!{(1,2) }*+{\bullet}="e"
			!{(3,2) }*+{\bullet}="f"
			!{(4.2,1.6) }*+{\bullet}="g"
			!{(0.3,-1.7) }*+{\bullet}="h"
			!{(5,-1.7) }*+{\bullet}="i"
"a":@/^0cm/"b"
"a":@/^0cm/@{~>}"e"  
"e":@/^0cm/@{~>}"f"  
"b":@/^0cm/"d"
"c":@/^0cm/"d"
"d":@/^0cm/@{~>}"a"
"c":@/^0.2cm/"i"
"i":@/^0.2cm/@{~>}"c"  
"f":@/^0cm/"g"
"g":@/^0cm/"c"
"h":@/^0cm/"a"
"e":@/^0cm/"b"
"d":@/^0cm/"h"
"c":@/^0cm/"f"
		}
	}
	\hfill
	\mbox{ \xygraph{
			!{<0cm,0cm>;<1cm,0cm>:<0cm,1cm>::}
			!{(0,0) }*+{\bullet}="a"
			!{(2,0.6) }*+{\bullet}="b"
			!{(4,0) }*+{\bullet}="c"
			!{(3,-1.3) }*+{\bullet}="d"
			!{(1,2) }*+{\bullet}="e"
			!{(3,2) }*+{\bullet}="f"
			!{(4.2,1.6) }*+{\bullet}="g"
			!{(0.3,-1.7) }*+{\bullet}="h"
			!{(5,-1.7) }*+{\bullet}="i"
			!{(1,2) }*={\xybox{*+=[blue]=<0.5cm,0.5cm>\frm{e}}}
			!{(3.7,1.2) }*={\xybox{*+=[blue]=<2cm,2.8cm>\frm{e}}}
			!{(1.3,-0.6) }*={\xybox{*+=[green]=<4cm,3cm>\frm{e}}}
			!{(5,-1.7) }*={\xybox{*+=[green]=<0.5cm,0.5cm>\frm{e}}}
			"a":@/^0cm/"b"
			"b":@/^0cm/"d"
			"c":@/^0cm/"d"
			"c":@/^0cm/"i"
			"f":@/^0cm/"g"
			"g":@/^0cm/"c"
			"h":@/^0cm/"a"
			"e":@/^0cm/"b"
			"d":@/^0cm/"h"
			"c":@/^0cm/"f"
		}
	}
\caption{On the left: a strongly connected digraph $(V,E)$, fast edges are represented by straight lines and slow edges (i.e., every edge that is not fast) by wavy lines. 
On the right: the fast subgraph $(V,E_F)$, equivalence classes are enclosed by colored circles. Each circle corresponds to a point of the associated DAG, green ones are the local minima.}
\label{primografo}
\end{figure}
In Section \ref{sec: iteration}, we will add some superscript to indicate that this is the first step of an iterative procedure: graphs constructed at iteration step $i$ will have a superscript $(i)$. 

We now proceed in two steps. First, we define a \emph{reduced dynamics} on the vertices of $V$ belonging to the minimal classes $M\in\mathcal M$; this corresponds to the trace process as defined for example in \cite{L1}. Then, using the reduced dynamics, we introduce an \emph{effective dynamics}, which is a Markov chain on $\mathcal M$.

\subsection{Trace process}

Let us denote by $\hat V$ the support of $\mathcal M$, i.e. the set of all vertices belonging to the minimal classes:
\begin{equation*}
    \hat V:=\bigcup_{M\in \mathcal M} M \,=\, \textrm{Su}\left(\mathcal M\right) \,\subseteq V\,.
\end{equation*}
It is worth noticing that $\hat V$ is the set of recurrent points in the fast graph $(V,E_F)$, while $V \setminus \hat V$ is the set of transient points in the fast graph.
\smallskip

We consider the continuous-time Markov chain $X_t^N$ with transition graph $(V,E,r_N)$ and following \cite{BL, LX, L1} we introduce the \emph{trace process} on $\hat V$. This is defined as
\begin{equation}
 \hat X^N_t:=X_{T^{-1}(t)}^N \,\in \hat{V} \;,
\end{equation}
where $T(t)$ is the amount of time spent in $\hat V$ up to instant $t$ and $T^{-1}(t)$ denotes its generalized inverse:
\begin{equation*}
 T(t):= \int_0^t \mathbb{I}\big(X_s^N \in \hat V\big) \,\d s \;, \qquad T^{-1}(t):=\sup\left\{s\geq0\,|\, T(s)\leq t\right\}\,.
\end{equation*}
$\hat X^N_t$ is again a continuous-time Markov chain with some useful properties \cite{L1}.

Let us consider the probability that, starting from the vertex $z \in V$, the original process $X_t^N$ enters the set $\hat V$ for the first time at the vertex $y\in \hat V$: 
\begin{equation}\label{probas}
	P_{N}^{\hat V}( y\,|\, z) \,:=\, \mathbb P\left(X^N_\theta=y \,\big|\, X^N_0=z\right)\,,\quad  \theta:=\inf\{t>0\,:\,X^N_t\in \hat V\}\,.
\end{equation}
When $z\in \hat V$, then $P_{N}^{\hat V}( y\,|\, z)$ is $0$ if $z\neq y$ and $1$ if $z=y$.
The trace process $\hat X^N_t$ has transition rates given by the following weights $\hat r_N$ \cite{BL, LX, L1}.

\begin{definition}
   For every $x,y \in \hat V$ with $x\neq y$ the \emph{reduced rates} is
\begin{equation}\label{reduced}
\begin{split}
    \hat{r}_N(x,y)\,:&=\, \sum_{z \in V }\,r_N(x,z)\; P_{N}^{\hat V}(y\,|\,z) \\
    &=\,
    r_N(x,y) \,+\sum_{z\in V\setminus\hat V} r_N(x,z)\; P_{N}^{\hat V}(y\,|\,z)\,.
\end{split}
\end{equation}
The \textit{reduced edge set} is $\hat E:=\{(x,y)\in \hat V\times\hat V \,:\, \hat r_N(x,y)>0 \,\}\,$ and we call $(\hat V,\hat E,\hat r_N)$ the \textit{reduced weighted graph}.
\end{definition}
The reduced graph $(\hat V, \hat E, \hat r_N)$ is strongly connected, since $(V,E)$ was assumed strongly connected.
We denote by $\hat{\pi}_N$ the associated invariant measure on $\hat V$.
\begin{figure}
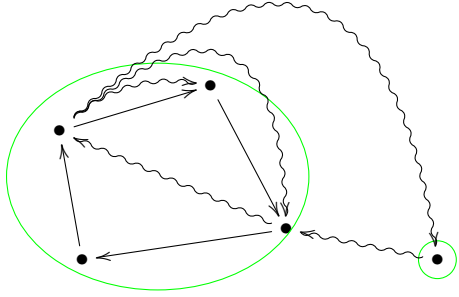

	\centering
	\mbox{ \xygraph{
			!{<0cm,0cm>;<1cm,0cm>:<0cm,1cm>::}
			!{(0,0) }*+{\bullet}="a"
			!{(2,0.6) }*+{\bullet}="b"
			!{(3,-1.3) }*+{\bullet}="d"
			!{(0.3,-1.7) }*+{\bullet}="h"
			!{(5,-1.7) }*+{\bullet}="i"
			!{(1.3,-0.6) }*={\xybox{*+=[green]=<4cm,3cm>\frm{e}}}
			!{(5,-1.7) }*={\xybox{*+=[green]=<0.5cm,0.5cm>\frm{e}}}
			"a":@/^0cm/"b"
			"b":@/^0cm/"d"
			"d":@/^0cm/@{~>}"a"
			"a":@/^2.6cm/@{~>}"i"
				"a":@/^0.3cm/@{~>}"b"
				"a":@/^1.8cm/@{~>}"d"
				"i":@/^0cm/@{~>}"d"
			"h":@/^0cm/"a"
			"d":@/^0cm/"h"
		}
	}
	\caption{The reduced graph $(\hat V,\hat E)$ obtained from the graph $(V,E)$ of Figure \ref{primografo}. The vertex set is the support of minimal classes, while the reduced edges represent the paths connecting two vertices on the original graph passing through the non-minimal classes.} 
\label{secondografo}
\end{figure}

\begin{remark}\label{slower}
When $x,y\in\hat V$ belong to different minimal classes $M,M'\in\mathcal M$, the rate $\hat r_N(x,y)$ is slower compared to the fast edges in $E_F$, namely $\lim_{N\to\infty}\frac{\hat r_N(x,y)}{r_N(e)}=0$ for $e\in E_F$. More precisely, it has velocity slower or equal to the class $E_{F-1}$.
Indeed by definition, the edges going out from the minimal class $M$ cannot be fast, then the result follows from the second expression in \eqref{reduced}.
\end{remark}

The hitting probabilities \eqref{probas}, when $z\not\in \hat V$, have also a representation in terms of directed spanning forests \cite{P}:
\begin{equation}
\label{Pitman}
    P_{N}^{\hat V}(y\,|\,z) \,=\, \frac{r_N\left(\mathcal F_{\hat V}(z\to y)\right)}{r_N\left(\mathcal F_{\hat V}\right)}\;,
\end{equation}
where $\mathcal{F}_{\hat V}$ is the set of directed spanning forests of $(V,E)$ composed by $|\hat V|$ non-spanning arborescences rooted at the vertices of $\hat V$, and $\mathcal{F}_{\hat V}(z \to y)$ is the subset of $\mathcal F_{\hat V}$ such that $z$ belongs to the component rooted at $y$.
Recall that the weight of a set of forests (or generally of any family of subgraphs) is defined as
\begin{equation}
    r_N(\mathcal{F}_{\hat V}) = \sum_{f \in \mathcal{F}_{\hat V}} \prod_{(x,y) \in f} r_N(x,y)\,.
\end{equation}
Note that every forest $f\in\mathcal F_{\hat V}$ has the same number of edges, that is $|V|-|\hat V|$.
Substituting \eqref{Pitman} into \eqref{reduced} yields, by direct inspection,
\begin{equation} \label{reduced1}
    \hat{r}_N(x,y) \,=\, 
     \frac{r_N\left(\mathcal F_{\hat V\setminus\{x\}}(x\to y)\right)}{r_N\left(\mathcal F_{\hat V}\right)} \,.
\end{equation}
\smallskip

By definition of trace process it is easy to see that, given $U_2\subset U_1\subset V$ it is equivalent to construct the trace process directly from $V$ to $U_2$ or to trace from $V$ to $U_1$ first and then take the trace from $U_1$ to $U_2$.
In particular, the trace process on $\hat V$ can be obtained by performing $|V|-|\hat V|$ successive single-node reductions.

When we reduce the process on $V\setminus\{s\}$ removing a single node $s$, we obtain rates
\begin{equation} \label{single_reduced}
	r_N'(x,y) \,=\, r_N(x,y) \,+\, \frac{r_N(x,s)\;r_N(s,y)}{\sum_{z\in V\setminus\{s\}}r_N(s,z)}
\end{equation}
for $x,y\in V\setminus\{s\}$, $x\neq y$. The latter is a simple application of \eqref{reduced} or \eqref{reduced1}. 
Such single-node reduction of the weighted graph is a generalized \textit{star-delta reduction} \cite{CPV,GH} and will be discussed in Section \ref{sec:stardelta}.
Removing one single node at a time for every node in $V\setminus \hat V$ we obtain the desired graph $(\hat V, \hat E, \hat r_N)$.
\smallskip

Now we show that Assumption \ref{FRL2} is preserved under the reduction operation.

\begin{proposition}\label{lollodorme}
Let $(V,E,r_N)$ be a strongly connected weighted graph satisfying Assumption \ref{FRL2}.
Then the reduced graph $(\hat V,\,\hat E,\, \hat r_N)$ defined by \eqref{reduced} satisfies Assumption \ref{FRL2}.
\end{proposition}

\begin{proof}
Let $\underline {\hat m}=({\hat m}_e)_{e\in \hat E}$, $\underline {\hat n}=({\hat n}_e)_{e\in \hat E}$ two vectors of non-negative integers, both having sum of entries equal to $K$. 
Using expression \eqref{reduced1}, simplifying the denominators and observing that the numerators $r_N\big(\mathcal F_{\hat V\setminus\{x\}}(x\to y)\big)$ are sums of products of $|V|-|\hat V|+1$ edge rates, we may write
\begin{equation}
	\frac{\prod_{e\in \hat E}\, \hat r_N(e)^{\hat m_e}}{\prod_{e\in \hat E}\, \hat r_N(e)^{\hat n_e}} \,=\,
\frac{\sum_{\underline m\in A}\,c(\underline m)\,\prod_{e\in E}\,r_N(e)^{m_e}}{\sum_{\underline n\in B}\,d(\underline n)\,\prod_{e\in E}\,r_N(e)^{n_e}} \,=:\rho
\end{equation}
for suitable sets $A,B$ of vectors $\underline m=(m_e\geq0)_{e\in E}$, $\underline n=(n_e\geq0)_{e\in E}$ respectively such that $\sum_{e\in E}m_e=\sum_{e\in E}n_e=K\,(|V|-|\hat V|+1)$ 
and suitable coefficients $c(\underline m),\,d(\underline n)>0$.
\smallskip

By Assumption \ref{FRL2} for the original rates $r_N$, there exist $\underline m^*\in A$ and $\underline n^*\in B$ such that
$\frac{\prod_{e\in E}r_N(e)^{m_e}}{\prod_{e\in E}r_N(e)^{m^*_e}}$ and  
$\frac{\prod_{e\in E}r_N(e)^{n_e}}{\prod_{e\in E}r_N(e)^{n_e^*}}$
have finite limits as $N\to\infty$ for every $\underline m\in A$ and $\underline n\in B$.
Therefore $\rho$ rewrites as
\begin{equation}
	\rho \,=\,
	\frac{c(\underline m^*)\,\prod_{e\in E}r_N(e)^{m^*_e}}{d(\underline n^*)\,\prod_{e\in E}r_N(e)^{n^*_e}}\;\ 
	 \frac{1+\sum_{\underline m\in A,\,\underline m\neq \underline m^*}\frac{c(\underline m)\,\prod_{e\in E}r_N(e)^{m_e}}{c(\underline m^*)\,\prod_{e\in E}r_N(e)^{m^*_e}}}
	 {1+\sum_{\underline n\in B,\,\underline n\neq \underline n^*}\frac{d(\underline n)\,\prod_{e\in E}r_N(e)^{n_e}}{d(\underline n^*)\,\prod_{e\in E}r_N(e)^{n^*_e}}} \;.
\end{equation}
Using Assumption \ref{FRL2}, the first ratio on the right-hand side has a limit as $N\to\infty$, and the second ratio converges to a finite positive value. Then $\rho$ has a limit.
%
\end{proof}

\subsection{Effective dynamics}
We are going to define the effective rates $\bar r_N$ of an irreducible Markov chain on the state space $\mathcal M$. We use the reduced rates as done in metastability \cite{BL, LX, L1}. 
\smallskip

Given a minimal class $M\in \mathcal M$, we denote $\mu_{M}^N$ the invariant measure of the Markov chain with transition graph $\left(M, E_F[M], r_N\right)$. This is the strongly connected subgraph of $(V,E_F)$ formed by  fast edges with both endpoints in $M$, their rates coincide with $r_N$. 
By the Markov chain tree theorem we have
\begin{equation} \label{markovchaintree4M}
\mu_{M}^N(x) \,=\, \frac{r_N\left(\mathcal T_x\left[M,E_F[M]\right]\right)}{r_N\left(\mathcal T\left[M,E_F[M]\right]\right)} \,,
\end{equation}
where $\mathcal T\left[M,E_F[M]\right]$ denotes the set of arborescences of the subgraph $\left(M, E_F[M]\right)$.
We also denote $\mu_{M}:=\lim_{N\to+\infty}\mu_M^N$ the corresponding limiting measure, that exists by the same argument of Proposition \ref{existence} using Assumption \ref{FRL}.

\smallskip

Motivated by the averaging principle, we define the effective rates by averaging the reduced rates over the fast dynamics:

\begin{definition}\label{defeff}
For every $M_1,M_2\in \mathcal M$ with $M_1\neq M_2$ the \textit{effective rate} is
\begin{equation}
    \label{ratesefficacissimi}
        \bar r_N(M_1,M_2) :=\sum_{y_1 \in M_1,\,y_2 \in M_2} \mu_{M_1}^N(y_1)\;\hat{r}_N(y_1,y_2) \,.
\end{equation}
The \textit{effective edge set} is $\bar E:=\{(M_1,M_2)\in\mathcal M\times \mathcal M \,:\, \bar r_N(M_1,M_2)>0 \}\,$ and we call $(\mathcal M, \bar E, \bar r_N)$ the \textit{effective weighted graph}.
This is the transition graph of an \textit{effective Markov chain} $\bar X^{N}_t$ on $\mathcal M$.
\end{definition}

The effective graph $(\mathcal M,\bar E,\bar r_N)$ is strongly connected, since the reduced graph $(\hat V, \hat E,\hat r_N)$ is so and each $\mu_M^N$ is strictly positive on $M$. We denote by $\bar\pi_N$ the associated invariant measure on  $\mathcal M$. 
\begin{figure}
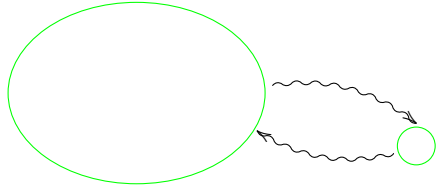

	\centering
	\mbox{ \xygraph{
				!{<0cm,0cm>;<1cm,0cm>:<0cm,1cm>::}
		!{(3.1,-0.9) }="x"
			!{(2.9,-1.5) }="p"
				!{(4.7,-1.8) }="q"
					!{(5,-1.4) }="b"
	!{(1.3,-1) }*={\xybox{*+=[green]=<3.4cm,2.4cm>\frm{e}}}
	!{(5,-1.7) }*={\xybox{*+=[green]=<0.5cm,0.5cm>\frm{e}}}
		"x":@/^0.22cm/@{~>}"b"
			"q":@/^0.2cm/@{~>}"p"
	}
} \vskip 4pt
\caption{The effective graph $(\mathcal M,\bar E)$ obtained from Figures \ref{primografo},  \ref{secondografo}. The nodes are the minimal classes. The effective edges are slower compared to the fast edges $E_F$ of the original graph $(V,E)$.}
\label{terzografo}
\end{figure}

\begin{remark}\label{slower2}
The effective rates are slower compared to the fast edges in $E_F$, namely $\lim_{N\to\infty}\frac{\bar r_N(M_1,M_2)}{r_N(e)}=0$ for every $(M_1,M_2)\in\bar E$ and $e\in E_F$. More precisely, they have velocity slower or equal to the edges of class $E_{F-1}$.
This follows by definition \eqref{ratesefficacissimi} and Remark \ref{slower}.
\end{remark}

\begin{proposition} \label{iodormo}
Let $(V,E,r_N)$ be a strongly connected weighted graph satisfying Assumption \ref{FRL2}.
Then the effective graph $(\mathcal M,\,\bar E,\, \bar r_N)$ defined by \eqref{ratesefficacissimi} also satisfies Assumption \ref{FRL2}.
\end{proposition}

\begin{proof}
Let $\underline {\bar m}=(\bar m_{\bar e})_{\bar e\in \bar E}\,$, $\underline {\bar n}=({\bar n}_{\bar e})_{\bar e\in \bar E}$ be two vectors of non-negative integers, both having sum of the entries equal to $K$. 
Using definition \eqref{ratesefficacissimi} of the effective rates, then writing the reduced rates $\hat r_N$ as \eqref{reduced1} and the invariant measures $\mu_M^N$ as \eqref{markovchaintree4M} we can simplify the common denominators and write the terms in the numerator of the form $\prod_{M\in\mathcal M}r_N\big(\mathcal T_x\left[M,E_F[M]\right]\!\big)\; r_N\big(\mathcal F_{\hat V\setminus\{y_1\}}(y_1\to y_2)\big)$ as sums of products of $|V|-|\mathcal M|+1$ edge rates obtaining:
\begin{equation}
	\frac{\prod_{\bar e\in \bar E}\, \bar r_N(\bar e)^{{\bar m}_{\bar e}}}{\prod_{{\bar e}\in \bar E}\, \bar r_N(\bar e)^{{\bar n}_{\bar e}}} \,=\,
\frac{\sum_{\underline m\in A}\,c(\underline m)\,\prod_{e\in E}\,r_N(e)^{m_e}}{\sum_{\underline n\in B}\,d(\underline n)\,\prod_{e\in E}\,r_N(e)^{n_e}}
\end{equation}
for suitable sets $A,B$ of vectors $\underline m=(m_e\geq0)_{e\in E}$, $\underline n=(n_e\geq0)_{e\in E}$ respectively such that $\sum_{e\in E}m_e=\sum_{e\in E}n_e=K\,\big(|V|-1\big)$ and coefficients $c(\underline m),\,d(\underline n)>0$.

Now the proof is analogous to the proof of Proposition \ref{lollodorme}.
\end{proof}

Assumption \ref{FRL2} is preserved for the effective graph, therefore the chain $\bar X^N_t$ is a multiscale Markov chain too. Then by the same argument of Proposition \ref{existence} the invariant measure $\bar\pi_N$ converges to a limiting measure, denoted by $\bar\pi$.
\smallskip

We are ready to state our main result

\begin{theorem}[Limiting invariant measure] \label{ilteo}
Let $X^N_t$ be a continuous time Markov chain with strongly connected transition graph $(V,E,r_N)$ satisfying Assumption \ref{FRL2}. 
Let $\pi_N$ be its unique invariant measure.
Then its limit exists, is a probability measure on $V$ and can be expressed as
\begin{equation}\label{laformula}
	\pi:=
	\lim_{N \rightarrow \infty} \pi_N =
	\sum_{M\in \mathcal M} \bar \pi(M)\, \mu_M \;.
\end{equation}
In particular $\pi(x)=0$ for $x\in V\setminus \hat V$ and more explicitly:
\begin{equation}\label{laformula2}
	\pi(x)= \begin{cases} 
	\,\bar \pi(M)\;\mu_M(x) & \textrm{, }\, x\in M  \textrm{ for }M\in\mathcal M\\
	\,0 & \textrm{, }\,x\in Q  \textrm{ for }Q\in\mathcal Q \setminus \mathcal M \end{cases} \ .
\end{equation}
\end{theorem}
\smallskip

\begin{remark} 
\label{classiveloci}
	For $M\in\mathcal M$, the weighted graph $\left(M, E_F\left[M\right], r_N\right)$ has all comparable weights.
Fix an arbitrary edge $(x^*,y^*)\in E_F\left[M\right]$ and rescale all the weights by a global factor: the rescaled weights $\frac{r_N(x,y)}{r_N(x^*,y^*)}$, $(x,y)\in E_F\left[M\right]$, produce the same invariant measure $\mu^N_M$. 
Since all the rates are comparable, the limits $\lim_{N\to + \infty}\frac{r_N(x,y)}{r_N(x^*,y^*)}=:r^*(x,y)>0$ exist and are finite and positive.
In this simple situation the limiting measure $\mu_M=\lim_{N\to\infty}\mu^N_M$ is the invariant measure of the chain with rates $r^*$.
\end{remark}
\smallskip

\section{Recursion for the limiting invariant measure and forest representation of the multiscale dynamics} \label{sec: iteration}

Remark \ref{classiveloci} shows how to compute the limiting measures $\mu_M$ with $M\in\mathcal M$. The measure $\bar\pi$ instead is again the limiting invariant measure of a multiscale Markov chain. Is it possible to close formula \eqref{laformula} and use it to compute the limiting measure $\pi\,$?

For this purpose we now illustrate the recursive structure hidden in formula \eqref{laformula}, which is geometrically encoded by a forest. This is the same introduced in \cite{BL, LX} to understand and describe the dynamic metastable behavior. Since the construction is based on an iteration of the procedure described in Section \ref{sec: main}, in the present section we adapt the notation to clarify the iterative step.

\subsection{Iterative notation}

We label by $(0)$ the initial strongly connected graph $(V,E,r_N)$ 
and by $(1)$ the strongly connected graph $\left(\mathcal M, \bar E, \bar r_N\right)$ obtained through the effective construction described in Section \ref{sec: main}. 
We adopt the following notation:
\begin{equation*}
\begin{array}{l}
\left(\mathcal M^{(0)}, \bar{E}^{(0)}, \bar r_N^{(0)}\right) :=\left(V, E, r_N\right)\,,\\[4pt]
\left(\mathcal M^{(1)}, \bar{E}^{(1)}, \bar r_N^{(1)}\right) :=\left(\mathcal M, \bar E, \bar r_N\right)\,.
\end{array}
\end{equation*}
We are going to define recursively for every index $i\in\mathbb N$ a strongly connected graph $\left(\mathcal M^{(i)}, \bar{E}^{(i)},\bar r_N^{(i)}\right)$ that is obtained from $\left(\mathcal M^{(i-1)}, \bar{E}^{(i-1)}, \bar r_N^{(i-1)}\right)$ by the same procedure giving $\left(\mathcal M^{(1)}, \bar{E}^{(1)},\bar r_N^{(1)}\right)$ from $\left(\mathcal M^{(0)}, \bar{E}^{(0)},\bar r_N^{(0)}\right)$. 
At each step $i\geq1$, we consider $\bar E^{(i-1)}_F$ the set of fastest edges in $\bar E^{(i-1)}$ according to the rates $\bar{r}_N^{(i-1)}$, we denote by $\mathcal Q^{(i)}$ the collection of equivalence classes of the set $\mathcal M^{(i-1)}$ with respect to communication on $\left(\mathcal M^{(i-1)}, \bar{E}^{(i-1)}_F\right)$, and by $\mathcal M^{(i)}$ the collection of the minimal equivalence classes. Finally we proceed as in Section \ref{sec: main} to define the new \textit{effective rates} $\bar r_N^{(i)}$.
\smallskip

This means that the Markov dynamics at iteration $(i)$ is obtained as an \textit{effective dynamics} on the minimal classes associated with fast communication at the previous iteration $(i-1)$. Note that applying iteratively Remark \ref{slower2}, the effective rates $\bar r_N^{(i)}$  have velocity slower or equal to $E_{F-i}$. 
\smallskip

The elements $M^{(0)}\in\mathcal M^{(0)}$ are simply the vertices of $V$, while elements $M^{(1)}\in\mathcal M^{(1)}$ are pairwise disjoint subsets of $V$.
%
This construction is iterated level by level: elements $M^{(i)}\in \mathcal M^{(i)}$ are pairwise disjoint subsets of $\mathcal M^{(i-1)}\,$ and an element $M^{(i-1)}\in\mathcal M^{(i-1)}$ may belong either to some minimal class $M^{(i)}\in\mathcal M^{(i)}$ or to some non-minimal class $Q^{(i)}\in \mathcal Q^{(i)}\setminus \mathcal M^{(i)}$.
Moreover, we can associate to the collection of elements $M^{(i)}\in \mathcal M^{(i)}$ a collection of pairwise disjoint subsets of $V$ defined recursively as
\begin{equation}
\textrm{Su}\big(M^{(i)}\big):=\cup_{M^{(i-1)}\in M^{(i)}}\, \textrm{Su}\big(M^{(i-1)}\big) \,\subseteq V \;,
\end{equation}
$\textrm{Su}\left(M^{(0)}\right) := \{x\}$ for $M^{(0)}=x\in V$.
Of course if $M^{(i-1)}\in M^{(i)}$ then $\textrm{Su}\left(M^{(i-1)}\right)\subseteq \textrm{Su}\left(M^{(i)}\right)$.
We also set
$\textrm{Su}\left(\mathcal M^{(i)}\right):=\bigcup_{M^{(i)}\in\mathcal M^{(i)}} \textrm{Su}\left( M^{(i)}\right)\,$.
\smallskip

An important fact is that applying iteratively Proposition \ref{iodormo}, if the rates $r_N=\bar r^{(0)}_N$ satisfy Assumption \ref{FRL2} then the rates $\bar{r}_N^{(i)}$ satisfy Assumption \ref{FRL2} at every step $i\geq0$.
\smallskip

We denote by  $\bar{\pi}_N^{(i)}$ the unique invariant measure of the Markov chain with transition graph $\left(\mathcal M^{(i)},\bar E^{(i)},\bar{r}_N^{(i)}\right)$. $\bar{\pi}_N^{(i)}$ is a probability measure on $\mathcal M^{(i)}$. 
We call $\bar{\pi}^{(i)}$ the corresponding limiting measure, which exists since Assumption \ref{FRL2} is satisfied at each level.
\smallskip

Given $M^{(i)}\in\mathcal M^{(i)}$ for $i\geq1$, we denote by $\mu_{M^{(i)}}^{(i-1),N}$ the invariant measure of the Markov chain with weighted transition graph $\left(M^{(i)}, \bar{E}_F^{(i-1)}\left[M^{(i)}\right], \bar{r}_N^{(i-1)}\right)$.
We remark once more that $\mu_{M^{(i)}}^{(i-1),N}$ can be thought either as a positive probability measure on $M^{(i)}$ or as a probability measure on $\mathcal M^{(i-1)}$ supported only on the class $M^{(i)}\in \mathcal M^{(i)}$.
We call $\mu_{M^{(i)}}^{(i-1)}$ the corresponding limiting measure (see Remark \ref{classiveloci}).

\subsection{Recursive formula}

Using the recursive notation we have introduced, formula \eqref{laformula} reads
\begin{equation}\label{laformulait}
	\bar{\pi}^{(0)} \,= \sum_{M^{(1)}\in \mathcal M^{(1)}} \bar \pi^{(1)}\big(M^{(1)}\big)\; \mu_{M^{(1)}}^{(0)} \,.
\end{equation}

Theorem \ref{ilteo} can be applied recursively at any step $i\geq1$, since the relation between the chains at iterations $(i-1)$ and $(i)$ is always the same and Assumption \ref{FRL2} is always satisfied.
This gives the following recursive formula:
\begin{equation}\label{laformulaiti}
	\bar{\pi}^{(i-1)}  \,= \sum_{M^{(i)}\in \mathcal M^{(i)}} \bar \pi^{(i)}\big(M^{(i)}\big)\; \mu_{M^{(i)}}^{(i-1)}
\end{equation}
for all $i\geq1$.
More explicitly:
\begin{equation}\label{laformulaiti2}
\begin{split}
	&\bar\pi^{(i-1)}\big(M^{(i-1)}\big) = \\[2pt] &= \begin{cases} 
	\;\bar \pi^{(i)}\left(M^{(i)}\right)\;\mu_{M^{(i)}}^{(i-1)}\left(M^{(i-1)}\right) & \textrm{, } M^{(i-1)}\in M^{(i)}  \textrm{ for }M^{(i)}\in\mathcal M^{(i)}\\[4pt]
	\;0 & \textrm{, } M^{(i-1)}\in Q^{(i)} \textrm{ for }Q^{(i)}\in\mathcal Q^{(i)} \setminus \mathcal M^{(i)} \end{cases} .
\end{split}
\end{equation}

\smallskip

The number of iterations is finite and smaller or equal to the number $F$ of different edge velocity classes; let us denote it by $k$. The iterative procedure stops at the first iteration $(k)$ for which $\left|\mathcal M^{(k)}\right|=1$, so that there is a single element $M^{(k)}\in \mathcal M^{(k)}$ and $\bar{\pi}^{(k)}\left(M^{(k)}\right)=1$.
Thus $k=\inf\left\{i\in\mathbb N\,:\, |\mathcal M^{(i)}|=1\right\}$.
The recursive procedure can be formally extended to further steps, but then we have $\left|\mathcal M^{(j)}\right|=1$ and $\bar{\pi}^{(j)}\left(M^{(j)}\right)=1$ for all $j\geq k$. 
\smallskip

Observe that iterating formula \eqref{laformulaiti2} we obtain $\bar \pi^{(0)}\left(M^{(0)}\right)=0$, unless $M^{(0)}\in M^{(1)}\in M^{(2)}\in\dots\in M^{(k)}$ for suitable $M^{(i)}\in\mathcal M^{(i)}$, $i=1\dots k-1$, and for the unique element $M^{(k)}\in\mathcal M^{(k)}$. In other terms:
\begin{equation} \label{zeromeasure}
 \bar \pi^{(0)}\big(M^{(0)}\big)=0 \quad \textrm{if }\,M^{(0)}\notin \mathrm{Su}\big(M^{(k)}\big)\,.
\end{equation}
On the other hand, $\bar{\pi}^{(k)}\left(M^{(k)}\right)=1$. 
It then follows that the recursion \eqref{laformulaiti} has a unique solution, which provides a formula for the limiting invariant measure $\pi=\bar{\pi}^{(0)}$ of Theorem \ref{ilteo}.

\subsection{Forest representation}
Using recursion \eqref{laformulaiti}, we obtain a representation formula for the limiting measure $\pi=\bar{\pi}^{(0)}$, in terms of weights of paths on a forest. 
The nodes of this forest are the elements of $\bigcup_{i=0}^k \mathcal M^{(i)}$. They are naturally organized into levels: the nodes at level $i$ are the elements of $\mathcal M^{(i)}$. 

Edges of the forest only connect nodes belonging to consecutive levels: precisely, we draw an edge from a node $M^{(i)}\in \mathcal M^{(i)}$ to a node $M^{(i-1)}\in \mathcal M^{(i-1)}$ if $M^{(i-1)}\in M^{(i)}$, or equivalently, if $\textrm{Su}\left(M^{(i-1)}\right)\subseteq \textrm{Su}\left(M^{(i)}\right)$. 
The edge $\left(M^{(i)}, M^{(i-1)}\right)$ is given weight $\bar{\mu}_{M^{(i)}}^{(i-1)}\left(M^{(i-1)}\right)$.
\smallskip

In this way, at most one edge can enter each node $M^{(i-1)}$: if $M^{(i-1)}\in M^{(i)}$ for some $M^{(i)}\in \mathcal M^{(i)}$, there is one edge entering node $M^{(i-1)}$ from node $M^{(i)}$; if instead $M^{(i-1)}\in Q^{(i)}$ for some $Q^{(i)}\in \mathcal Q^{(i)}\setminus \mathcal M^{(i)}$, there is no edge entering $M^{(i-1)}$. There are no edges entering the unique element $M^{(k)}\in\mathcal M^{(k)}$, since there are no nodes at higher levels.
This construction produces a weighted directed forest. For each component of the forest, the unique node at the highest level is called the \emph{root}. In particular, there is one single component whose root is $M^{(k)}$ at highest level $k$.
For each node except for the leaves, the sum of the weights going out from the node is equal to one, since $\sum_{M^{(i-1)}\in  M^{(i)}}\bar{\mu}_{M^{(i)}}^{(i-1)}\left(M^{(i-1)}\right)=1$.
\smallskip

Given $M^{(0)}\in \mathcal M^{(0)}$, there exists a unique path going from the root of its component to $M^{(0)}$.
If this root is not $M^{(k)}$, then $\bar{\pi}^{(0)}\left(M^{(0)}\right)=0$ by \eqref{zeromeasure}.
If the root is $M^{(k)}$, which occurs when $M^{(0)}\in \textrm{Su}\left(M^{(k)}\right)$, then we call $\gamma_{M^{(0)}}$ the unique path from the root $M^{(k)}$ to $M^{(0)}$.
We then have the following formula
\begin{theorem}\label{ilteo2}
	The limiting measure \eqref{laformula} in Theorem \ref{ilteo} is the unique solution $\pi=\bar\pi^{(0)}$ of the recursion \eqref{laformulaiti}. It can be computed as:
\begin{equation}\label{tantiprod}
\begin{split}
	&\bar{\pi}^{(0)}\big(M^{(0)}\big) =\\ &=\begin{cases}
		\;\prod_{\left(M^{(i)},M^{(i-1)}\right)\,\in\,\gamma_{M^{(0)}}} \,\mu^{(i-1)}_{M^{(i)}}\left(M^{(i-1)}\right) & \textrm{if }\, M^{(0)}\in \textup{Su}\left(M^{(k)}\right)\\[4pt]
		\;0 & \textrm{if }\, M^{(0)}\notin \textup{Su}\left(M^{(k)}\right)
	\end{cases}
\end{split}
\end{equation}
for all $M^{(0)}=x\in V=\mathcal M^{(0)}\,$.
\end{theorem}
The proof of the above Theorem consists just in iterating the recursion \eqref{laformulaiti2} $k$ times starting from $\bar\pi^{(k)}(M^{(k)})=1$, and assigning weight zero to all the roots less deep than $k$ by \eqref{zeromeasure}.
The construction and the formula \eqref{tantiprod} are illustrated in Figure \ref{fiore} .

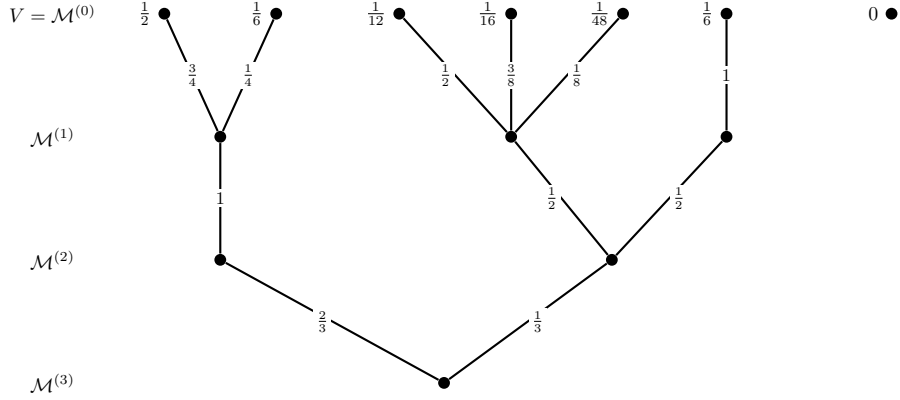
\begin{figure}[h]
	\centering
	\begin{tikzpicture}[
		scale=0.74, transform shape,
		level distance=2.2cm,
		dot/.style={circle, fill=black, minimum size=6pt, inner sep=0pt},
		edge from parent/.style={draw, thick},
		edge label/.style={midway, fill=white, inner sep=1.5pt, font=\small}
		]
		
		\begin{scope}[xshift=-10cm]
			\node at (0,0) {$\mathcal{M}^{(3)}$};
			\node at (0,2.2) {$\mathcal{M}^{(2)}$};
			\node at (0,4.4) {$\mathcal{M}^{(1)}$};
			\node at (0,6.6) {$V=\mathcal{M}^{(0)}$};
		\end{scope}
		
		\node[dot] (root) at (-3,0) {}
		child[grow=up, xshift=-4cm] {
			node[dot] {} 
			child {
				node[dot] {} 
				child[sibling distance=2cm] { node[dot, label=left:$\frac{1}{6}$] {} edge from parent node[edge label] {$\frac{1}{4}$} }
				child[sibling distance=2cm] { node[dot, label=left:$\frac{1}{2}$] {} edge from parent node[edge label] {$\frac{3}{4}$} }
				edge from parent node[edge label] {1}
			}
			edge from parent node[edge label, left] {$\frac{2}{3}$}
		}
		child[grow=up, xshift=3cm] {
			node[dot] {}
			child[xshift=-2.8cm, sibling distance=2cm] {
				node[dot] {}
				child { node[dot, label=left:$\frac{1}{48}$] {} edge from parent node[edge label, right] {$\frac{1}{8}$} }
				child { node[dot, label=left:$\frac{1}{16}$] {} edge from parent node[edge label] {$\frac{3}{8}$} }
				child { node[dot, label=left:$\frac{1}{12}$] {} edge from parent node[edge label, left] {$\frac{1}{2}$} }
				edge from parent node[edge label, left] {$\frac{1}{2}$}
			}
			child[xshift=2.8cm] {
				node[dot] {}
				child { node[dot, label=left:$\frac{1}{6}$] {} edge from parent node[edge label] {1} } 
				edge from parent node[edge label, right] {$\frac{1}{2}$}
			}
			edge from parent node[edge label, right] {$\frac{1}{3}$}
		};
		\node[dot, label=left:$0$] at (5,6.6) {};
\end{tikzpicture}
\caption{Forest representation of a multiscale dynamics with $k=3$ levels. The weight of each edge $(M^{(i)},M^{(i-1)})$ between consecutive layers is $\bar\mu^{(i-1)}_{M^{(i)}}\left(M^{(i-1)}\right)$. The limiting invariant measure $\pi=\bar\pi^{(0)}$ of the whole dynamics is computed according to formula \eqref{tantiprod} and written next to the leaves.}
\label{fiore}
\end{figure}
\smallskip

\section{Generalized star-delta reduction on directed graphs}\label{sec:stardelta}

In this section we discuss the properties of arborescences in a network reduction that is the directed analogous of the star-delta reduction on electrical networks \cite{BF,GH}.
The results here are independent of the multiscale structure of the weights, so we omit the dependence of $r$ on the parameter $N$. 

\subsection{1-vertex reduction} We start discussing the reduction removing a single vertex from a complete weighted graph.
\smallskip

Let $K_{n+1}=(V,E,r)$ be a complete weighted digraph on $n+1$ vertices $V=\hat V\cup\{s\}\,$, $|\hat V|=n$. The special vertex $s$ is the one we are going to remove. A non-negative rate $r(e) \geq 0$ is associated to every edge $e\in E=\{(x,y)\in V\times V\,:\, x\neq y\}$. 
We split the edge set $E$ as the disjoint union of two subsets $E= E^\textup{B}\cup E^\textup{I}$ (\textit{boundary edges} and \textit{internal edges}):

\begin{itemize}
	\item $E^\textup{B}$ consists of edges of type $(x,x')$ with $x,x' \in \hat V \,$, $x\neq x'\,$;
	\item $E^\textup{I}$ consists of edges of type $(x,s)$ or $(s,x)$ with $x\in \hat V \,$.
\end{itemize}

\begin{figure}[h]
	\centering
	\begin{tikzpicture}[scale=0.95,
		vertex/.style={circle, fill=black, inner sep=1.5pt},
		every edge/.style={draw, thick, ->}
		]
		\node[vertex, label=above:$s$] (s) at (0,0) {};
		
		\def\n{6}    
		\def\r{1.8}  
		
		\foreach \i in {1,...,\n}{
			\node[vertex, label={\i*60-90}:{\small $x_{\i}$}] (x\i) at ({360/\n * (\i-1)}:\r) {};
			
			\draw[->, thick, shorten >=3pt, shorten <=3pt] (s) -- (x\i);
			\draw[->, thick, shorten >=3pt, shorten <=3pt] (x\i) -- (s);
			
			\draw[dashed, gray, ->] (x\i) -- ({360/\n * (\i-1)}:{\r+0.6});       
			\draw[dashed, gray, ->] (x\i) -- ({360/\n * (\i-1) + 6.5}:{\r+0.6});  
			\draw[dashed, gray, ->] (x\i) -- ({360/\n * (\i-1) - 6.5}:{\r+0.6});  
		}
	\end{tikzpicture}
	\hspace{25pt}
	\begin{tikzpicture}[scale=0.95,
		vertex/.style={circle, fill=black, inner sep=1.5pt},
		every edge/.style={draw, thick, ->}
		]
		\def\n{6}    
		\def\r{1.8}  
		
		\foreach \i in {1,...,\n}{
			\node[vertex, label={\i*60-90}:{\small $x_{\i}$}] (x\i) at ({360/\n * (\i-1)}:\r) {};
		}
		
		\foreach \i in {1,...,\n}{
			\foreach \j in {1,...,\n}{
				\ifnum\i<\j
				\draw[->, thick, shorten >=3pt, shorten <=3pt] (x\i) -- (x\j);
				\draw[->, thick, shorten >=3pt, shorten <=3pt] (x\j) -- (x\i);
				\fi
			}
		}
		
		\foreach \i in {1,...,\n}{
			\draw[dashed, gray, ->] (x\i) -- ({360/\n * (\i-1)}:{\r+0.7});       
			\draw[dashed, gray, ->] (x\i) -- ({360/\n * (\i-1) + 7}:{\r+0.7});  
			\draw[dashed, gray, ->] (x\i) -- ({360/\n * (\i-1) - 7}:{\r+0.7});  
		}
		
\end{tikzpicture}
\caption{On the left: the internal edges $E^\textup{I}$ of the star graph $K_{n+1}$. On the right: the edges $E^*$ of the reduced multigraph $\hat{\mathbb K}_n$. Dashed arrows going out from the vertices suggest the presence of boundary edges $E^\textup{B}$. In this picture we have $n=6$.}
\label{fig:grafo-stella}
\end{figure}
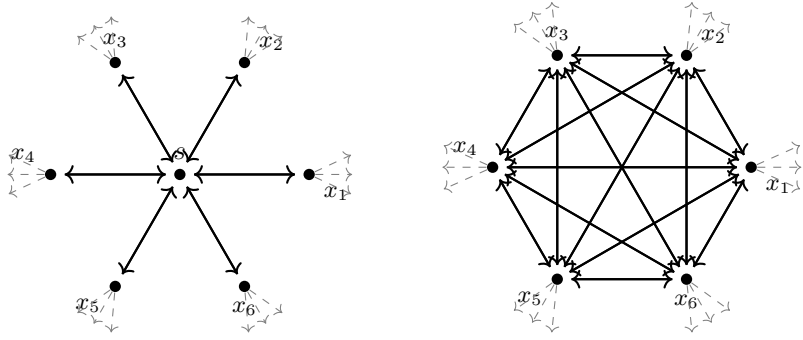

Now, we define the \textit{generalized star-delta reduction of} $K_{n+1}$ as the weighted complete digraph $\hat K_n=(\hat{V},\hat{E},\hat r)$ with vertex set $\hat{V}=V \setminus \{s\}$, edge set $\hat{E}:=E^\textup{B}$, and reduced weights as in \eqref{single_reduced}:
\begin{equation}\label{spli}
	\hat{r}(x,x') := 
r(x,x') + r^*(x,x') \,,\quad x,x'\in\hat V,\,x\neq x' 
\end{equation}
where
\begin{equation} \label{r*}
r^*(x,x') := \begin{cases}
\;\dfrac{r(x,s)\,r(s,x')}{r^\textup{out}(s)} &,\,\textrm{if }r^\textup{out}(s)\neq 0\\
\;0 &,\,\textrm{if }r^\textup{out}(s)=0
\end{cases}
\end{equation}
and	$r^\textup{out}(s):=\sum_{x\in\hat V} r(s,x)\,$ is the total weight going out from $s$.

\begin{remark}\label{remo}
It will be sometimes natural to replace $\hat K_n$ by a weighted multigraph $\hat{\mathbb K}_n$ containing two different copies of each edge. Every pair of vertices $x, x' \in \hat V$, $x\neq x'$, is connected by an edge in $E^\textup{B}$ with weight $r(x,x')$ and by another edge in $E^*$ (a distinct copy of $E^\textup{B}$) with weight $r^*(x,x')$ defined by \eqref{r*}. We denote by $\hat{\mathbb R}$ the weights on the multiedge set $\hat{\mathbb E}:=E^\textup{B}\cup E^*$ defined in this way. Thus $\hat{\mathbb K}_n=(\hat V, \hat{\mathbb E}, \hat{\mathbb R})$.
\end{remark}

Let us fix a vertex $x_1\in\hat V$. We denote respectively by $\mathcal T_{x_1}$, $\hat{\mathcal T}_{x_1}$, $\hat \T_{x_1}$ the sets of arborescences rooted at $x_1$ of the graphs $K_{n+1}$, $\hat K_n$ and the multigraph $\hat{\mathbb K}_n$. 
\smallskip

It is useful to classify every spanning arborescence $\tau\in \mathcal T_{x_1}$ according to the structure of 
\begin{equation}
\hat f:=\tau \cap E^\textup{B} = \tau \smallsetminus E^\textup{I} \,.
\end{equation}
$\hat f$ contains only edges of $E^\textup{B}$ and it is a directed spanning forest of the reduced graph $\hat K_n$. In general $\hat f$ has $h$ components, one always directed torwards $x_1$.
We call them $(C_l)_{l=1}^h$, i.e. $\hat f=\cup_{l=1}^hC_l$, and we call $C_1$ the component rooted at $x_1$.
\smallskip

Observe that $\tau$ always contains one edge of $E^\textup{I}$ going from $s$ to $C_1$. The number $h$ of components of $\hat f$ ranges from $1$ to $n$, precisely: $h=1$ when $\tau$ contains only one edge of $E^\textup{I}$ (which goes out from $s$), $h=n$ when $\tau$ contains only edges of $E^\textup{I}$ (one going from $s$ to $x_1$ and $n-1$ going from $x_2,\dots,x_n$ to $s$).  In general, $\hat f$ has $h$ components if the arborescence $\tau$ on the original graph contains one edge of $E^\textup{I}$ exiting from $s$, $h-1$ entering in $s$, and $n-h$ edges of $E^\textup{B}$. So $h = |\tau\cap E^\textup{I}|$. 
\smallskip

Similarly, every spanning arborescence $t \in\hat\T_{x_1}$ can be classified according to the structure of 
\begin{equation}
\hat f:=t\cap E^\textup{B} =  t \smallsetminus E^*\,. 
\end{equation}
Also in this case $\hat f$ is a directed spanning forest of $\hat K_n$. It has $h$ components $(C_l)_{l=1}^h$ and without loss of generality $C_1$ is rooted at $x_1$.
The number $h$ of components ranges from $1$ to $n$. We have $h=1$ when $t$ contains no edge of $E^*$, while $h=n$ when $t$ contains only edges of $E^*$. In general, $\hat f$ has $h$ components if the arborescence $t$ on the multigraph contains $h-1$ edges of $E^*$ and $n-h$ edges of $E^\textup{B}$.
\smallskip


For $h=1,\dots,n$ and $x_2,\dots,x_h\in\hat V\setminus\{x_1\}$ all distinct, we denote by $\hat {\mathcal F}_{\{x_1,x_2\dots,x_h\}}$ the set of spanning forests of $\hat K_n$ with $h$ components rooted at $x_1, x_2,\dots, x_h$.
Given $\hat f\in\hat {\mathcal F}_{\{x_1,x_2,\dots,x_h\}}\,$, we denote  by $\mathcal T_{x_1}(\hat f)$, $\hat \T_{x_1}(\hat f)$ respectively the subsets of arborescences $\tau\in\mathcal T_{x_1}\,$, $\tau\in\hat\T_{x_1}$ such that $\tau\cap E^\textup{B}=\hat f$.

\smallskip

The following result is a generalization of \cite{CPV} for directed graphs. In the un-directed case the relation holds true for un-directed spanning trees without fixing a root; in the directed case it is necessary to fix a root.

\begin{proposition} \label{Lemmarnout}
The weight of the set $\mathcal T_{x_1}$ of arborescences of $K_{n+1}$ and the reduced weight of the set $\hat{\mathcal T}_{x_1}$ of arborescences of $\hat K_{n}$ are related by:
\begin{equation} \label{eq:lemmarnout}
	r\left(\mathcal{T}_{x_1}\right) \,=\, r^\textup{out}(s)\,\ \hat{r}\left(\hat{\mathcal{T}}_{x_1}\right) \,.
\end{equation}
\end{proposition}

\begin{proof}
The case $r^\textup{out}(s)=0$ is trivial, thus without loss of generality we assume $r^\textup{out}(s)>0$.
Let us compute $r(\mathcal T_{x_1})$ by separating different contributions according to the previous classification:
\begin{equation} \label{class_sum}
	r(\mathcal T_{x_1}) \,=\, 
	 \sum_{k=1}^n \, \sum_{\{x_2,\dots,x_h\}\in P^{h-1}} \,\sum_{\hat f\in\hat{\mathcal F}_{\{x_1,x_2,\dots,x_h\}}} \,r\left(\mathcal T_{x_1}(\hat f)\right)
\end{equation}
where $P^{h-1}$ denotes the collection of subsets of $\hat V\setminus\{x_1\}$ with $h-1$ distinct vertices.
\smallskip

\textit{Case $h=1$}. Let $\hat f\in \hat{\mathcal T}_{x_1}=\hat{\mathcal F}_{x_1}$. Every arborescence $\tau\in\mathcal T_{x_1}$ with  $\tau\cap E^\textup{B}=\hat f\,$ has the form $\tau=\hat f\,\cup (s,y)$ for some $y\in \hat V$. Then:
\begin{equation} \label{per1}
r\left(\mathcal T_{x_1}(\hat f)\right) = \sum_{\tau\in\mathcal T_{x_1}(\hat f)} r(\tau) \,=\, r(\hat f)\,\sum_{y\in\hat V}r(s,y) \,=\, r(\hat f)\;r^\textup{out}(s)\,.
\end{equation}
\begin{figure}[h]
	\centering
	\begin{tikzpicture}[scale=0.95,
		vertex/.style={circle, fill=black, inner sep=1.5pt},
		every edge/.style={draw, thick, ->}
		]
		\node[vertex, label=above:$s$] (s) at (0,0) {};
		
		\def\n{6}    
		\def\r{1.8}  
		
		
		\node[vertex, label=right:{$x_1$}] (x1) at ({360/\n * (1-1)}:\r) {};
		\node[vertex] (x2) at ({360/\n * (2-1)}:\r) {};
		\node[vertex] (x3) at ({360/\n * (3-1)}:\r) {};
		\node[vertex] (x4) at ({360/\n * (4-1)}:\r) {};
		\node[vertex] (x5) at ({360/\n * (5-1)}:\r) {};
		\node[vertex, label=right:{$y$}] (x6) at ({360/\n * (6-1)}:\r) {};
		
		\draw[->, thick, shorten >=2pt, shorten <=2pt] (s) -- (x6);
		
		\draw[dashed, gray, ->, shorten >=2pt, shorten <=2pt] (x4) -- (x3);
		\draw[dashed, gray, ->, shorten >=2pt, shorten <=2pt] (x3) -- (x2);
		
		\draw[dashed, gray, ->, shorten >=2pt, shorten <=2pt] (x5) -- (x6);
		\draw[dashed, gray, ->, shorten >=2pt, shorten <=2pt] (x6) -- (x1);
		\draw[dashed, gray, ->, shorten >=2pt, shorten <=2pt] (x2) -- (x1);
		
	\end{tikzpicture}
	\hspace{40pt}
	\begin{tikzpicture}[scale=1,
		vertex/.style={circle, fill=black, inner sep=1.5pt},
		every edge/.style={draw, thick, ->}
		]
		
		\def\n{6}    
		\def\r{1.8}  
		
		
		\node[vertex, label=right:{$x_1$}] (x1) at ({360/\n * (1-1)}:\r) {};
		\node[vertex] (x2) at ({360/\n * (2-1)}:\r) {};
		\node[vertex] (x3) at ({360/\n * (3-1)}:\r) {};
		\node[vertex] (x4) at ({360/\n * (4-1)}:\r) {};
		\node[vertex] (x5) at ({360/\n * (5-1)}:\r) {};
		\node[vertex] (x6) at ({360/\n * (6-1)}:\r) {};

		\draw[dashed, gray, ->, shorten >=2pt, shorten <=2pt] (x4) -- (x3);
		\draw[dashed, gray, ->, shorten >=2pt, shorten <=2pt] (x3) -- (x2);
		
		\draw[dashed, gray, ->, shorten >=2pt, shorten <=2pt] (x5) -- (x6);
		\draw[dashed, gray, ->, shorten >=2pt, shorten <=2pt] (x6) -- (x1);
		\draw[dashed, gray, ->, shorten >=2pt, shorten <=2pt] (x2) -- (x1);
		
\end{tikzpicture}
\caption{On the left: an arborescence $\tau\in\mathcal T_{x_1}(\hat f)$ on $K_{n+1}$. On the right: an arborescence $t\in \mathbb T_{x_1}(\hat f)$ on the reduced multigraph $\hat{\mathbb K}_n$. In both cases the corresponding forest $\hat{f}$ obtained by keeping only the boundary edges (dashed ones) is the same and has a unique component.}
\label{esempio tau e f}
\end{figure}
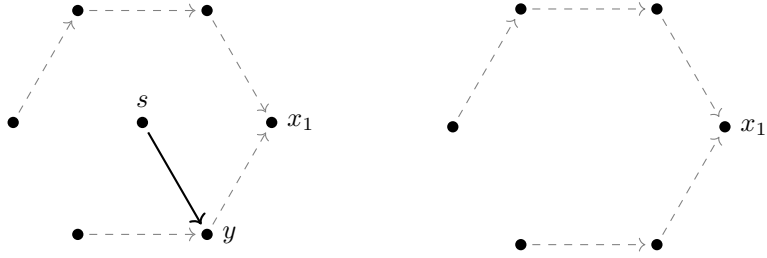
\smallskip
	
\textit{Case $h=2$}. Let $\hat f\in \hat{\mathcal F}_{\{x_1, x_2\}}\,$, $x_2\in \hat V\setminus\{x_1\}$. Every $\tau\in \mathcal T_{x_1}$ with $\tau\cap E^\textup{B}=\hat f\,$ has the form $\tau=\hat f\,\cup (x_2,s) \cup (s,y)$ for some $y\in C_1$. 
Then:
\begin{equation} \label{per2}
	r\left(\mathcal T_{x_1}(\hat f)\right) =\,  r(\hat f)\; r(x_2,s)\, \sum_{y\in C_1} r(s,y) 
\end{equation}
where we recall that $C_1=C_1(\hat f)$ is the component of $\hat f$ containing $x_1$.
\smallskip

\textit{Case $h\leq n$}. Let $\hat f\in \hat{\mathcal F}_{\{x_1, x_2,\dots,x_h\}}\,$, $x_2,\dots,x_h\in \hat V\setminus\{x_1\}$ distinct points. 
Reasoning as before, in the general case we obtain:
\begin{equation}\label{perk}
	r\left(\mathcal T_{x_1}(\hat f)\right) \,=\, r(\hat f)\,\; r(x_2,s)\cdots r(x_h,s)\,\sum_{y\in C_1}r(s,y)
\end{equation}
%
\smallskip
	
Now, in order to compute $\hat r\left(\hat{\mathcal T}_{x_1}\right)$ it is convenient to use the multigraph $\mathbb{K}_n$ introduced in Remark \ref{remo}, since:
\begin{equation} \label{multi_uguale}
\begin{split}
	\hat r\left(\hat{\mathcal T}_{x_1}\right) &\,=
	\sum_{\hat\tau\in\hat{\mathcal T}_{x_1}} \prod_{e\in\hat \tau}\,\hat r(e) \,= \sum_{\hat\tau\in\hat{\mathcal T}_{x_1}} \prod_{e\in\hat \tau}\left(r(e)+r^*(e)\right) =\\[2pt]
&= \sum_{\hat\tau\in\hat{\mathcal T}_{x_1}} \sum_{\{A,B\}}\, \prod_{e\in A}r(e)\,\prod_{e\in B}r^*(e) \,= \sum_{t\in\hat{\T}_{x_1}} \prod_{e\in t}\,\mathbb R(e) \,=\, \hat{\mathbb R}\left(\hat{\mathbb T}_{x_1}\right) 
\end{split}
\end{equation} 
where the sum over $\{A,B\}$ runs over the partitions in two sets of the edges of $\hat\tau$.
Using identity \eqref{multi_uguale} and the classification introduced before the proposition:
\begin{equation}\label{class2_sum}
\begin{split}
\hat{r}\left(\hat{\mathcal T}_{x_1}\right) \,&=\, \hat{\mathbb R} \left(\hat{\mathbb T}_{x_1}\right) \,=\, 
\sum_{k=1}^n \, \sum_{\{x_2,\dots,x_k\}\in P^{k-1}} \,\sum_{\hat f\in\hat{\mathcal F}_{\{x_1,x_2,\dots,x_k\}}} \,\hat{\mathbb R}\left(\hat\T_{x_1}(\hat f)\right)
\end{split}
\end{equation} 
\smallskip

\textit{Case $h=1$}. Let $\hat f\in \hat{\mathcal T}_{x_1}\equiv\hat{\mathcal F}_{x_1}$. The only arborescence $t\in\hat{\mathbb T}_{x_1}$ with $t\cap E^\textup{B}=\hat f$ is $t=\hat f$ itself. In the multigraph the edges in $E^\textup{B}$ are weighted by $r$, hence
\begin{equation} \label{peR1}
	\hat{\mathbb R}\left(\hat{\mathbb T}_{x_1}(\hat f)\right) \,=\, 
r(\hat f)\,.
\end{equation}
\smallskip
	
\textit{Case $h=2$}. Let $\hat f\in \hat{\mathcal F}_{\{x_1, x_2\}}\,$, $x_2\in \hat V\setminus\{x_1\}$. Every $t\in \hat \T_{x_1}$ with $t\cap E^\textup{B}=\hat f\,$ has the form $t=\hat f\,\cup (x_2,y)^*$ for some $y\in C_1$. Then:
\begin{equation} \label{peR2}
	\hat{\mathbb R}\left(\hat\T_{x_1}(\hat f)\right) \,=\, r(\hat f)\; \sum_{y\in C_1}r^*(x_2,y) \,.
\end{equation}
\smallskip

\textit{Case $h\leq n$}. Let $\hat f\in \hat{\mathcal F}_{\{x_1, x_2,\dots,x_h\}}\,$, $x_2,\dots,x_h\in \hat V\setminus\{x_1\}$ distinct points. In the general case we obtain:
\begin{equation}\label{peRk}
	\hat{\mathbb R}\left(\hat{\mathbb T}_{x_1}(\hat f)\right) \,=\;
 r(\hat f) \sum_{(y_2,\dots ,y_h)\in T^h}\,r^*(x_2,y_2)\cdots r^*(x_h,y_h)
\end{equation}
where $T^h=T^h(C_1,\dots,C_h)=T^h(\hat f)$ denotes the set of all $(y_2,\dots ,y_h)\in(\hat V)^{h-1}$ ($y_{\ell}$'s are not necessarily distinct) such that 
the graph having as vertices the $h$ components $C_1,\dots,C_h$ of $\hat f$ and having an edge from $C_\ell$ to $C_m$ iff $y_{\ell}\in C_m$ is an arborescence directed torward $C_1$ (see Figure \ref{esempio componenti foresta}).
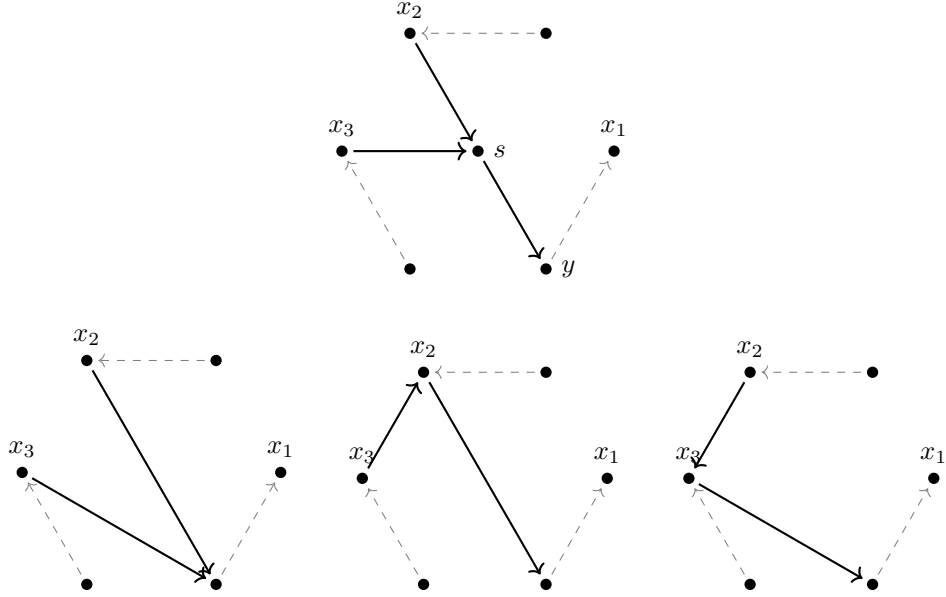
\begin{figure}[h]
	\centering
	\begin{tikzpicture}[scale=1,
		vertex/.style={circle, fill=black, inner sep=1.5pt},
		every edge/.style={draw, thick, ->},
		comp/.style={draw, blue, dashed, ellipse, inner sep=6pt}
		]
		\node[vertex, label=right:$s$] (s) at (0,0) {};
		\def\n{6}    
		\def\r{1.8}  
		\node[vertex, label=above:{$x_1$}] (x1) at ({360/\n * (1-1)}:\r) {};
		\node[vertex] (x2) at ({360/\n * (2-1)}:\r) {};
		\node[vertex, label=above:{$x_2$}] (x3) at ({360/\n * (3-1)}:\r) {};
		\node[vertex, label=above:{$x_3$}] (x4) at ({360/\n * (4-1)}:\r) {};
		\node[vertex] (x5) at ({360/\n * (5-1)}:\r) {};
		\node[vertex, label=right:{$y$}] (x6) at ({360/\n * (6-1)}:\r) {};
	
		\draw[->, thick, shorten >=2pt, shorten <=2pt] (s) -- (x6);
		\draw[->, thick, shorten >=2pt, shorten <=2pt] (x4) -- (s);
		\draw[->, thick, shorten >=2pt, shorten <=2pt] (x3) -- (s);
		
		\draw[dashed, gray, ->, shorten >=2pt, shorten <=2pt] (x2) -- (x3);
		
		\draw[dashed, gray, ->, shorten >=2pt, shorten <=2pt] (x6) -- (x1);
		\draw[dashed, gray, ->, shorten >=2pt, shorten <=2pt] (x5) -- (x4);
		
		
	\end{tikzpicture}
	\vskip 12pt
	\begin{tikzpicture}[scale=0.95,
		vertex/.style={circle, fill=black, inner sep=1.5pt},
		every edge/.style={draw, thick, ->},
		comp/.style={draw, blue, dashed, ellipse, inner sep=6pt}
		]
		\def\n{6}    
		\def\r{1.8}  
		\node[vertex, label=above:{$x_1$}] (x1) at ({360/\n * (1-1)}:\r) {};
		\node[vertex] (x2) at ({360/\n * (2-1)}:\r) {};
		\node[vertex, label=above:{$x_2$}] (x3) at ({360/\n * (3-1)}:\r) {};
		\node[vertex, label=above:{$x_3$}] (x4) at ({360/\n * (4-1)}:\r) {};
		\node[vertex] (x5) at ({360/\n * (5-1)}:\r) {};
		\node[vertex] (x6) at ({360/\n * (6-1)}:\r) {};
		
		\draw[dashed, gray, ->, shorten >=2pt, shorten <=2pt] (x2) -- (x3);
		
		\draw[dashed, gray, ->, shorten >=2pt, shorten <=2pt] (x6) -- (x1);
		\draw[dashed, gray, ->, shorten >=2pt, shorten <=2pt] (x5) -- (x4);
		
		\draw[->, thick, black, shorten >=2pt, shorten <=2pt] (x4) -- (x6);
		\draw[->, thick, black, shorten >=2pt, shorten <=2pt] (x3) -- (x6);
		
		
	\end{tikzpicture}
	\hfill
	\begin{tikzpicture}[scale=0.9,
		vertex/.style={circle, fill=black, inner sep=1.5pt},
		every edge/.style={draw, thick, ->},
		comp/.style={draw, blue, dashed, ellipse, inner sep=6pt}
		]
		\def\n{6}    
		\def\r{1.8}  
		\node[vertex, label=above:{$x_1$}] (x1) at ({360/\n * (1-1)}:\r) {};
		\node[vertex] (x2) at ({360/\n * (2-1)}:\r) {};
		\node[vertex, label=above:{$x_2$}] (x3) at ({360/\n * (3-1)}:\r) {};
		\node[vertex, label=above:{$x_3$}] (x4) at ({360/\n * (4-1)}:\r) {};
		\node[vertex] (x5) at ({360/\n * (5-1)}:\r) {};
		\node[vertex] (x6) at ({360/\n * (6-1)}:\r) {};
		
		\draw[dashed, gray, ->, shorten >=2pt, shorten <=2pt] (x2) -- (x3);
		
		\draw[dashed, gray, ->, shorten >=2pt, shorten <=2pt] (x6) -- (x1);
		\draw[dashed, gray, ->, shorten >=2pt, shorten <=2pt] (x5) -- (x4);
		
		\draw[->, thick, black, shorten >=2pt, shorten <=2pt] (x4) -- (x3);
		\draw[->, thick, black, shorten >=2pt, shorten <=2pt] (x3) -- (x6);
		
	\end{tikzpicture}
	\hfill
	\begin{tikzpicture}[scale=0.9,
		vertex/.style={circle, fill=black, inner sep=1.5pt},
		every edge/.style={draw, thick, ->},
		comp/.style={draw, blue, dashed, ellipse, inner sep=6pt}
		]
		\def\n{6}    
		\def\r{1.8}  
		\node[vertex, label=above:{$x_1$}] (x1) at ({360/\n * (1-1)}:\r) {};
		\node[vertex] (x2) at ({360/\n * (2-1)}:\r) {};
		\node[vertex, label=above:{$x_2$}] (x3) at ({360/\n * (3-1)}:\r) {};
		\node[vertex, , label=above:{$x_3$}] (x4) at ({360/\n * (4-1)}:\r) {};
		\node[vertex] (x5) at ({360/\n * (5-1)}:\r) {};
		\node[vertex] (x6) at ({360/\n * (6-1)}:\r) {};
		
		\draw[dashed, gray, ->, shorten >=2pt, shorten <=2pt] (x2) -- (x3);
		
		\draw[dashed, gray, ->, shorten >=2pt, shorten <=2pt] (x6) -- (x1);
		\draw[dashed, gray, ->, shorten >=2pt, shorten <=2pt] (x5) -- (x4);
		
		\draw[->, thick, black, shorten >=2pt, shorten <=2pt] (x4) -- (x6);
		\draw[->, thick, black, shorten >=2pt, shorten <=2pt] (x3) -- (x4);
		
		
\end{tikzpicture}
\caption{Above: an arborescence $\tau\in\mathcal T_{x_1}(\hat f)$ on the original graph $K_{n+1}$. Below: three different arborescences $t_1,t_2,t_3\in \mathbb T_{x_1}(\hat f)$ on the reduced multigraph $\hat{\mathbb K}_n$. The forest $\hat{f}$ obtained by keeping only the boundary edges (dashed edges) is the same in each case and has $h=3$ components rooted at $x_1,x_2,x_3$. 
On the multigraph there are several ways of connecting these components.}
\label{esempio componenti foresta}
\end{figure}
Recalling the form of the rates $r^*$ from \eqref{r*}, the previous identity rewrites as
\begin{equation}\label{peRk_}
	\hat{\mathbb R}\left(\hat{\mathbb T}_{x_1}(\hat f)\right) \,=\,
r(\hat f)\;\ \frac{r(x_2,s)\cdots r(x_h,s)}{\left(r^\textup{out}(s)\right)^{h-1}} \sum_{(y_2,\dots ,y_k)\in T^h}\,r(s,y_2)\cdots r(s,y_h)
\end{equation}
\smallskip

By the following Lemma \ref{quandoquando}, comparing \eqref{per1}-\eqref{peR1} and \eqref{perk}-\eqref{peRk_} we obtain
\begin{equation}
 r\left(\mathcal T_{x_1}(\hat f)\right) \,=\, r^\textup{out}(s)\,\ \hat{\mathbb R}\left(\hat{\mathbb T}_{x_1}(\hat f)\right)
\end{equation}
for each spanning forest $\hat f$ of the reduce graph $\hat K_n$ with a root at $x_1$.
Finally, by \eqref{class_sum} and \eqref{class2_sum}, the latter identity implies \eqref{eq:lemmarnout}.
\end{proof}
\smallskip

\begin{lemma}\label{quandoquando}
$2\leq h\leq n$. Let $C_1,\dots,C_h$ be a partition of $\hat V$ and $T^h=T^h(C_1,\dots,C_n)$ defined as after \eqref{peRk}. Then:
\begin{equation}\label{let}
	\sum_{(y_2,\dots ,y_h)\in T^h}r(s,y_2)\cdots r(s,y_h) 	\,=\, 
\left(r^\textup{out}(s)\right)^{h-2}\;\sum_{y\in C_1}r(s,y)
\end{equation}
\end{lemma}

\begin{proof}
Consider a complete digraph $A$ with nodes $C_1,\dots,C_h$ and weight of every edge $(C_{\ell},C_m)$, $\ell\neq m$, equal to $w(C_{\ell},C_m):=\sum_{y\in C_m}r(s,y)=: w_m$. The weight of an edge depends only on the arrival node. Note also that $\sum_{m=1}^h w_m = r^{out}(s)$.
Let $\mathcal T^A_{C_1}$ be the set of arborescences of $A$ rooted at $C_1$. The left hand side of \eqref{let} is $w\left(\mathcal T^A_{C_1}\right)$.

By the directed version of the matrix tree theorem, $w\left(\mathcal T^A_{C_1}\right)$ can be computed as the determinant of a Laplacian matrix \cite{JJ}. 
Precisely, consider a $h\times h$ matrix $L$ indexed by the vertices $C_{\ell}$.
The diagonal entries are $L_{\ell,\ell}=\sum_{m\neq \ell}w(C_{\ell},C_m) = r^\textup{out}(s)-w_{\ell}$ and the off-diagonal entries are $L_{\ell,m}=-w(C_{\ell},C_m)=-w_m\,$. 
Let $L^1$ be the $(h-1)\times (h-1)$ matrix obtained by removing the first row and column from $L$. The directed version of the matrix tree theorem \cite{JJ} states that $w\left(\mathcal T^A_{C_1}\right)=\det{L^1}$.
Observe that $L^1$ is the difference of a diagonal and a rank-one matrix:
\begin{equation}\label{special}
	L^1 \,=\, r^\textup{out}(s)\,I \,-\, \underline 1\,w^T\,,
\end{equation}
where $I$ is the $(h-1)\times(h-1)$ identity matrix, $\underline 1$ is the $h-1$ dimensional column vector having all entries equal to one, and $w^T$ is the $h-1$ dimensional row vector given by $w^T=(w_2, w_3, \dots , w_h)$.
	The determinant of $L^1$ can be easily computed due to the special form \eqref{special}. 
	In particular $\underline 1$ is an eigenvector of $L^1$ with eigenvalue $w_1$. If we consider a basis of the $h-2$ dimensional hyperplane orthogonal to $w$, we obtain $h-2$ eigenvectors with eigenvalues all equal to $r^\textup{out}(s)$. Then
\begin{equation}
	\det{L^1} \,=\,\left(r^\textup{out}(s)\right)^{h-2}\,w_1 \,,
\end{equation}
which coincides with the right hand side of \eqref{let}.
This concludes the proof.
\end{proof}

\subsection{General case} 
Consider as usual $(V,E,r)$ a finite strongly connected digraph where $V$ is partitioned as $\hat V\cup S$.
As discussed in Section \ref{sec: main}, 
we can introduce a \textit{trace Markov chain} with transition graph $(\hat V,\hat E,\hat r)$.
The reduced rates $\hat r$ are defined by \eqref{reduced} and rewritten in combinatorial form as \eqref{reduced1}.
The following is a generalization to the directed case of a result of \cite{CPV} and extends Proposition \ref{Lemmarnout}.

\begin{theorem} \label{teo:treesrelation}
Let $x \in \hat V$. The weight of the set $\mathcal T_{x}$ of arborescences of $(V,E)$ and the reduced weight of the set $\hat{\mathcal T}_{x}$ of arborescences of $(\hat V,\hat E)$ are related by:
\begin{equation} \label{treesrelation}
	r\left(\mathcal{T}_x\right) = r\left(\mathcal{F}_{\hat V}\right)\, \hat{r}\left(\hat{\mathcal{T}}_x\right)
\end{equation}
where $\mathcal{F}_{\hat V}$ denotes the collection of spanning forests of $(V,E)$ with root set $\hat V$. 
\end{theorem}

\begin{proof}
We use induction on the number of removed vertices.

First of all, Proposition \ref{Lemmarnout} was written for a complete graph but zero weights were allowed, so it is valid for any graph: it corresponds to the statement of this theorem in the case $|S|=1$. 

Now, assume that \eqref{treesrelation} holds after the removal of $|S|-1$ nodes and prove that the statement also holds for $S$.
Let $y\in S$ and consider $\widetilde S:=S\setminus\{y\}\,$, $\widetilde V:=V\setminus\widetilde S=\hat V\cup \{y\}\,$. Denote by $\widetilde r$, $\widetilde {\mathcal T}_x$ the corresponding reduced rates and arborescences from $V$ to $\widetilde V$.
By the inductive hypothesis
\begin{equation} \label{ciaoGiulia}
	r\left(\mathcal T_x\right) \,=\, r\left(\mathcal F_{\widetilde V}\right)\, \widetilde{r}\left(\widetilde {\mathcal T}_x\right)\,.
\end{equation}
$(\hat V, \hat r)$ can be seen both as a $|S|$-nodes reduction starting from $(V,r)$ or as a $1$-node reduction starting from $(\widetilde V,\widetilde r)$. As discussed in Section \ref{sec: main}, the equivalence of these two constructions is clear from the trace process definition. 
We now apply Proposition \ref{Lemmarnout} to the $1$-node reduction from $\widetilde V$ to $\hat V$:
\begin{equation} \label{joe}
	\widetilde{r}\left(\widetilde{\mathcal T}_x\right) \,=	\, \bigg(\sum_{z\in \hat V}\widetilde{r}(y,z)\bigg)\; \hat{r}\left(\hat {\mathcal T}_x\right)\,.
\end{equation}
On the other hand, by adapting formula \eqref{reduced1} to the reduced rates $\widetilde r$ we have:
\begin{equation} \label{nonmiserve}
	\widetilde{r}(y,z) \,=\, \frac{r\left(\mathcal{F}_{\hat V}(y \rightarrow z)\right)}{r\left(\mathcal{F}_{\widetilde V}\right)} \,.
\end{equation}
Inserting \eqref{nonmiserve} into \eqref{joe} into \eqref{ciaoGiulia}, we conclude \eqref{treesrelation}. 
\end{proof}
\smallskip

\section{Proof of the main theorem} \label{sec:proof}

In this section we prove the main Theorem \ref{ilteo} from which the final formula \eqref{tantiprod} for the limiting invariant measure was deduced. 
We recall that $(V,E,r_N)$ is a strongly connected graph satisfying Assumption \ref{FRL2} and $\pi_N$ is its unique invariant measure.
$E_F$ is the subset of fast edges, $Q\in\mathcal Q$ are the equivalence classes of $V$ with respect to fast communication, which form a DAG, and $M\in\mathcal M$ are the local minima of this DAG. $\hat V\subseteq V$ denotes the support of $\mathcal M$.
$\mathcal T$ is the set of arborescences of $(V,E)$ and $\mathcal T^F$ denotes the subset of fast arborescences.

The validity of the  next statements extends to each step of the iteration described in Section \ref{sec: iteration}, since at each step we have a multiscale Markov chain.

\subsection{Lemmas about fast arborescences}

\begin{lemma} \label{Lemma1}
\begin{equation}
\lim_{N\to\infty}\pi_N(z)=0\quad \forall\,z\in V\setminus\hat V \,.
\end{equation}
\end{lemma}

\begin{proof}
	Let $z\in V\setminus\hat V$. By the Markov chain tree theorem, we need to prove
\begin{equation}
	\lim_{N \rightarrow \infty} \frac{r_N\left(\mathcal{T}_z\right)}{ r_N\left(\mathcal T\right)} =0 \,.
\end{equation}
Since numerator and denominator are finite sums of positive contributions, it suffices to prove that for every $\tau_z\in \mathcal T_z$ there exists $\tau_*\in \mathcal T$ such that
\begin{equation} \label{eqLemma1}
	\lim_{N \rightarrow \infty} \frac{r_N(\tau_z)}{ r_N(\tau_*)} =0 \,.
\end{equation}
	Let $Q\in \mathcal Q\setminus \mathcal M$ be the non-minimal class containing $z$. By definition there exists a fast edge exiting from $Q$, we call it $(w,w')\in E_F$ with $w\in Q$, $w'\notin Q$.  
\smallskip
	
	First assume $w=z$. Let $\tau_z\in\mathcal T_z$.
	We consider the graph $\tau_z\cup(z,w')\,$: this is a unicyclic graph where $z$ and $w'$ belong to the unique cycle. This cycle must contain at least one edge $e$ not in $E_F$ (otherwise $w'$ would also belong to the equivalence class $Q$). Removing this edge from the unicyclic graph we obtain a new arborescence $\tau_*\in\mathcal T$. And we have:
\begin{equation}
	\lim_{N\to\infty} \frac{r_N(\tau_z)}{r_N(\tau_*)} \,=\, \lim_{N\to\infty} \frac{r_N(e)}{r_N(z,w')} \,=\, 0 \,,
\end{equation}
since the edge $(z,w')$ is fast while $e$ is not and all the other edges are not changed. 
	
	Now suppose $w\neq z$. Since $z,w\in Q$, there exists a fast path from $z$ to $w$ in $E_F$. Denote it by $(z_0=z,\,z_1,\,z_2,\dots,z_n=w)$ and start by considering the unicyclic graph $\tau_z\cup(z,z_1)$. If all the edges in the unique cycle are fast we remove the edge in the cycle exiting from $z_1$ and we obtain a new arborescence $\tau_1\in\mathcal T_{z_1}$. We have
$\lim_{N\to\infty} \frac{r_N(\tau_z)}{r_N(\tau_1)} \,=\, 1\, 
$,
since $\tau_z$, $\tau_1$ have the same number of edges of each velocity.
We repeat the procedure for the next unicyclic graph $\tau_1\cup(z_1,z_2)$ and we continue until we obtain a unicyclic graph $\tau_{i_*-1}\cup(z_{i_*-1},z_{i_*})$ where the unique cycle contains at least one edge not in $E_F$ (at the latest this happens for $i_*=n$, as previously seen). Then we remove this edge and obtain an arborescence $\tau_{i_*}$ such that
\begin{equation} \begin{split}
	\lim_{N\to\infty} \frac{r_N(\tau_z)}{r_N(\tau_{i_*})} \,=\, \lim_{N\to\infty} \frac{r_N(\tau_z)}{r_N(\tau_1)}\; \frac{r_N(\tau_1)}{r_N(\tau_2)} \cdots \frac{r_N(\tau_{i_*-1})}{r_N(\tau_{i_*})} \,=\, 0 \,. \qedhere
\end{split} \end{equation}
\end{proof}

The above lemma says that for any $z\in V\setminus \hat V$ we have $\mathcal T_z\cap \mathcal T^F=\emptyset$. Therefore, all arborescences of this type can be neglected in the computation of the limit of $\pi_N$ (Proposition \ref{existence}).
We now consider another class of arborescences that has empty intersection with $\mathcal T^F$.
We call \textit{slow edges} all the edges of $E\setminus E_F$.

\begin{lemma}
	\label{Lemma2} Let $\tau\in \mathcal T$. Let $M\in \mathcal M$ be a minimal class. If at least one of the following properties is verified then $\tau\not\in \mathcal T^F$: 
\begin{itemize}	
\item[i)] $\tau$ is rooted at $x\in M$ 
and contains an edge going out from $M$;
\item[ii)] 
$\tau$ contains two or more edges going out from $M$;
\item[iii)] 
$\tau$ contains a slow edge connecting two points of $M$.
\end{itemize}
\end{lemma}

\begin{proof}
	It suffices to show that, given an arborescence $\tau\in\mathcal T$ having at least one of the features described above, it is possible to construct another arborescence $\tau_*\in\mathcal T$ such that \eqref{eqLemma1} holds true.
The constructions are similar.
\smallskip

i) By hypothesis there exists an edge $(z,y)\in\tau$ with $z\in M$ and $y\notin M$.  By definition of $M$, all the edges going out from $M$ are necessarily slow, i.e., $(z,y)\in E\setminus E_F$. Of course no edge of $\tau$ exits from the root $x$.

On the other hand the graph $(M,E_F[M])$ is strongly connected. So we can remove from $\tau$ all the edges starting from points of $M$ and replace them by an arborescence of $(M,E_F[M])$ rooted at $x$. In this way we obtain a new arborescence $\tau_*\in\mathcal T_x\,$. $\tau_*$ verifies \eqref{eqLemma1} since we removed $|M|-1$ edges, at least one slow, and we replaced them by the same number of edges, all fast.
\smallskip

ii) Suppose $\tau$ is not rooted inside class $M$ (otherwise we fall in the previous case). By hypothesis there exist two edges $(z,y),\,(z',y')\in\tau$ with $z,z'\in M$ distinct and $y,y'\notin M$. $(z,y),\,(z',y')$ must be slow.
We remove from $\tau$ all the edges starting from points of $M$ except for $(z,y)$, and we replace them by an arborescence of $(M,E_F[M])$ rooted at $z$. As before, we obtain a new arborescence $\tau_*\in\mathcal T$ that verifies \eqref{eqLemma1}.
\smallskip

iii) Suppose either $\tau$ is rooted at $x\in M$ and no edge of $\tau$ goes out from $M$ or $\tau$ is rooted elsewhere and only one edge of $\tau$ goes out from $M$ (otherwise we fall in i) or ii)). In the latter case we call $(x,y)\in\tau$ with $x\in M$, $y\notin M$ the unique edge going out from $M$.

By hypothesis $\tau$ contains a slow edge $(z,z')$ with both endpoints in $M$. 
We remove from $\tau$ all the edges starting from points of $M\setminus\{x\}$ and we replace them by an arborescence of $(M,E_F[M])$ rooted at $x$. We obtain a new arborescence $\tau_*\in\mathcal T$ that verifies \eqref{eqLemma1}.
\end{proof}

Consider a fast arborescence $\tau\in\mathcal T^F$. By Lemma \ref{Lemma2}, restricting $\tau$ to each minimal class $M\in\mathcal M$ we obtain an arborescence of the fast subgraph $\left(M, E_F\left[M\right]\right)$ (note that restricting to an arbitrary subgraph would produce a directed forest instead).
Moreover from the root of each sub-arborescence $\tau[M]$ it starts a single (necessarily slow) edge reaching a different class (minimal or not). The only exception is the class $M$ containing the root of $\tau$, since no edge of $\tau$ can leave this class.

We denote by $\mathcal T^F[M]$ the set of fast arborescences of $\mathcal T^F$ restricted to a class $M\in\mathcal M$.
By Lemma \ref{Lemma2}, they are simply characterized as set of arborescences of $\left(M, E_F\left[M\right]\right)$. In symbols; $\mathcal T^F[M]=\mathcal T\left[M, E_F\right]\,$.


\subsection{Lemmas about reduced weights}

We recall from Section \ref{sec: main} that $(\hat V, \hat E,\hat r_N)$ is the reduced graph, $(\mathcal M, \bar E, \bar r_N)$ is the effective graph, $\hat\pi_N$, $\bar \pi_N$ are the respective invariant measures.
We denote by $\hat{\mathcal T}$ the set of arborescences on $(\hat V, \hat E)$ and by $\hat{\mathcal T}^F$ the fast subset of $\hat{\mathcal T}$.
We denote by $\bar{\mathcal T}$ the set of arborescences of $\left(\mathcal M, \bar E\right)$.
\smallskip

The following lemma was established in \cite{BL, LX} and is one of the main motivations for the introduction of the trace process. A simple proof can be obtained by the ergodic theorem and Lemma \ref{Lemma1}. We give instead a combinatorial proof based on the results of Section \ref{sec:stardelta}.

\begin{lemma}\label{spanningreduced}
\begin{equation} \label{teoergodico}
	\hat{\pi}_N(x) \,=\, \frac{\pi_N(x)}{c_N} \quad \forall\,x\in \hat V
\end{equation}
and the normalization $c_N=\sum_{z\in \hat V}\pi_N(z)$ converges to $1$ as $N\to\infty$.
\end{lemma}

\begin{proof}
	By Theorem \ref{teo:treesrelation}, we have $r_N  (\mathcal{T}_x) = r_N\!\left(\mathcal{F}_{\hat V}\right)\,\hat{r}_N(\hat{\mathcal{T}}_x)$ for every $x\in \hat V$ and summing over $x\in \hat V$ we get $\sum_{z\in \hat V}r_N (\mathcal{T}_z) = r_N\!\left(\mathcal{F}_{\hat V}\right)\,\hat{r}_N(\hat{\mathcal{T}})\,$.
	Therefore from the Markov chain tree theorem we deduce:
\begin{equation}
	\hat\pi_N(x) \,=\, \frac{\hat r_N(\hat{\mathcal T}_x)}{\hat r_N(\hat{\mathcal T})} \,=\, \frac{r_N\left(\mathcal T_x\right)\,/\,r_N\left(\mathcal{F}_{\hat V}\right)}{\sum_{z\in \hat V}r_N\left(\mathcal T_z\right)\,/\,r_N\left(\mathcal{F}_{\hat V}\right)} \,=\, \frac{\pi_N(x)}{c_N}
\end{equation}
for every $x\in \hat V$ where $c_N= \frac{\sum_{z \in \hat{V}} r_N(\mathcal{T}_z)}{r_N(\mathcal{T})}\,$.
	By Lemma \ref{Lemma1}, we have that $c_N\to1$. 
\end{proof}
\smallskip


By Lemma \ref{spanningreduced}, then using formula \eqref{limF} that  also holds for the reduced rates, we obtain for every $x\in\hat V$
\begin{equation}
\label{duue}
    \lim_{N\to\infty} \pi_N(x) \,=\,
    \lim_{N\to\infty} \hat \pi_N(x) \,=\, 
    \lim_{N\to \infty} \frac{\hat r_N(\hat{\mathcal T}^F_x)}{\hat r_N(\hat{\mathcal T }^F)}\;.
\end{equation}
We need therefore to identify the fast reduced arborescences. The following lemma characterizes the fast reduced edges allowing to extend Lemma \ref{Lemma2}.

\begin{lemma}\label{fastat}
	Fast edges of the reduced graph $(\hat V, \hat E\,)$ coincide with fast edges of the original graph $(V,E)$ with both endpoints in $\hat V$, namely $\hat E_F=E_F[\hat V]\,$.
	
	Moreover, if $(x,y)\in \hat E_F$ necessarily both endpoints belong to the same minimal class $M\in\mathcal M$ and
\begin{equation} \label{eq:fastat}
	\hat r_N(x,y) \,=\, r_N(x,y) \,(1+ o_N(1))\,,\quad (x,y)\in \hat E_F
\end{equation}
where $o_N(1)$ vanishes as $N\to\infty$.	
\end{lemma}

\begin{proof}
	We make use of the forest representation \eqref{reduced1} for the reduced rates $\hat r_N$.
	
	First of all, we observe that by definition of $\hat V$ (support of the minimal classes of the DAG), there exists a forest of $\mathcal F_{\hat V}$ that contains only edges in $E_F$. This means that every fast forest in $\mathcal F_{\hat V}^F$ contains only edges in $E_F$. 
	
	Now let $x,y\in \hat V$ such that $(x,y)\in E_F$ and note that both endpoints belong to the same minimal class $M\in\mathcal M$.
Any forest of $\mathcal F_{\hat V\setminus\{x\}}(x\to y)$ is obtained starting from a forest of $\mathcal F_{\hat V}$ and adding a path in $V\setminus \hat V$ from $x$ to $y$. In particular starting by any element of $\mathcal F_{\hat V}^F$ and adding the fast edge $(x,y)$ we obtain a forest of $\mathcal F_{\hat V\setminus\{x\}}^F(x\to y)$ containing only fast edges (every fast forest of $\mathcal F_{\hat V\setminus\{x\}}^F(x\to y)$ is obtained in this way, since any edge $(x,z)\in E$ with $z\in V\setminus\hat V$ is necessarily slow). 
From \eqref{reduced1} it follows that $(x,y)\in \hat E_F$. 
Moreover \eqref{reduced1}, by an argument similar to \eqref{limF}, gives
\begin{equation}
	\hat r_N(x,y) \,=\, \frac{r_N\left(\mathcal F_{\hat V\setminus\{x\}}^F(x\to y)\right)}{r_N\big(\mathcal F_{\hat V}^F\big)}\,(1+ o_N(1)) \,=\, r_N(x,y)\,(1+ o_N(1)) \,.
\end{equation}
	
	On the other hand, consider $x,y\in \hat V$ such that $(x,y)\notin E_F$. Every edge $(x,z)\in E$ with $z\in V\setminus\hat V$ is slow. Therefore every forest of $\mathcal F_{\hat V\setminus\{x\}}(x\to y)$ must contain an edge of $E\setminus E_F$. From \eqref{reduced1} it follows that $(x,y)\notin \hat E_F$.
\end{proof}

From the fact that $\hat E_F$ is just $E_F[\hat V]$ we deduce the following lemma analogous to Lemma \ref{Lemma2}, hence we do not discuss the details.

\begin{lemma} \label{Lemma2-red}
 Let $\hat \tau\in \hat{\mathcal T}^F$. Let $M\in \mathcal M$ be a minimal class. The following properties are satisfied:
\begin{itemize}	
\item[i)] $\hat\tau$ contains no edges going out from $M$, if $\hat \tau$ is rooted inside $M$;
\item[ii)] $\hat\tau$ contains only one edge (necessarily slow) going out from $M$, if $\hat \tau$ is not rooted inside $M$;
\item[iii)] $\hat \tau$ contains no slow edges connecting two vertices of $M$.
\end{itemize}
\end{lemma}
\smallskip

\subsection{Proof of Theorem \ref{ilteo}}

First of all, Lemma \ref{Lemma1} guarantees that $\lim_{N\to\infty}\pi_N(x)$ is zero for every $x\in V\setminus \hat V$.

Now let $x\in\hat V$ and compute $\lim_{N\to\infty}\pi_N(x)$ via formula \eqref{duue}. Let $M\in\mathcal M$ be the minimal class containing $x$. Then, using Lemma \ref{Lemma2-red} for the characterization of fast reduced arborescences and identity \eqref{eq:fastat} for fast reduced rates, we may express the numerator of \eqref{duue} as: 
\begin{equation} \label{quuattro}
	\hat{r}_N\big(\hat{\mathcal T}_x^F\big) \,\approx\, r_N\!\left(\mathcal T_x^F[M]\right) \sum_{\bar{\tau} \in \bar {\mathcal T}_M} \prod_{(M_i,M_j) \in \bar{\tau}}\, \sum_{y_i \in M_i,\, y_j \in M_j } \hat{r}_N(y_i,y_j)\;r_N\!\left(\mathcal{T}_{y_i}^F[M_i]\right)
\end{equation}
where we recall that $\mathcal T^F_x[M]$ coincides with the set of arborescences of $(M,E_F[M])$ rooted at $x$, while $\bar{\mathcal T}_{M}$ denotes the set of arborescences of the effective graph $(\mathcal M,\bar E)$ rooted at $M$. The symbol $a_N\approx b_N$ is a shorthand for $a_N = b_N\,(1+o_N(1))$ with $o_N(1)$ vanishing as $N\to\infty$.
\smallskip

Let $\mu_M^N$ denote the invariant measure of $(M,E_F[M])$. Given $M_i,M_j\in \mathcal M$, by definition of effective rates \eqref{ratesefficacissimi} and by the Markov chain tree theorem for $\mu_M^N$, we have
\begin{equation} \label{hatbar}
    \bar r_N\left(M_i,M_j\right)\; r_N\left(\mathcal{T}^F\!\left[M_i\right]\right) \,= \sum_{y_i \in M_i,\, y_j \in M_j} \hat{r}_N(y_i,y_j)\; r_N\!\left(\mathcal{T}^F_{y_i}[M_i\right])\,.
\end{equation}
Using \eqref{hatbar} we rewrite \eqref{quuattro} as:
\begin{equation}
	\hat{r}_N\big(\hat{\mathcal T}_x^F\big) \,\approx\, r_N\!\left(\mathcal T_x^F[M]\right) \,\sum_{\bar{\tau} \in \bar{\mathcal T}_{M}} \prod_{(M_i,M_j) \in \bar{\tau}}\!\Big(\bar r_N\left(M_i,M_j\right)\; r_N\!\left(\mathcal{T}^F[M_i]\right) \Big) \,.
\end{equation}
In each arborescence $\bar{\tau} \in \bar{\mathcal T}_{M}$ there is exactly one effective edge going out from every class $M'\in\mathcal M$, $M'\neq M$. Thus the product of weights $r_N\!\left(\mathcal{T}^F[M_i]\right)$ goes out of the sum and the remaining sum is simply $\bar r_N\!\left(\bar{\mathcal T}_{M}\right)$, hence we obtain
\begin{equation}
\begin{split}
  	\hat{r}_N\big(\hat{\mathcal T}_x^F\big)  \,\approx\, r_N\!\left(\mathcal T_x^F\left[M\right]\right)\; \bar r_N\!\left(\bar{\mathcal T}_{M}\right) \prod_{M'\in\mathcal M\setminus\{M\}} r_N\!\left(\mathcal{T}^F[M']\right) \,.
\end{split}
\end{equation}
We can multiply and divide by $r_N\big(\mathcal T^F[M]\big)$ and using the Markov chain tree theorem for $\mu_M^N$ we find
\begin{equation} \label{inpiu}
\begin{split}
  	\hat{r}_N\big(\hat{\mathcal T}_x^F\big)  \,\approx\, \mu_M^N(x)\,\; \bar r_N\!\left(\bar{\mathcal T}_{M}\right) \prod_{M'\in\mathcal M} r_N\!\left(\mathcal{T}^F[M']\right) \;.
\end{split}
\end{equation}
\smallskip

By applying the same steps also to the denominator of \eqref{duue}, since $\hat V$ is the union of all minimal classes $M''$ and $\sum_{z\in M''}\,\mu_{M''}^N(z)$ equals $1$ for each $M''$, we obtain:
\begin{equation} \label{inpiu_den}
	\hat{r}_N\big(\hat{\mathcal T}^F\big)\,\approx\, \sum_{M''\in\mathcal M}\,
	\bar r_N\!\left(\bar{\mathcal T}_{M''}\right) \prod_{M'\in\mathcal M} r_N\!\left(\mathcal{T}^F[M']\right) \;.
\end{equation}
Taking the ratio of \eqref{inpiu}, \eqref{inpiu_den} the product over $M'$ cancels and we get
\begin{equation}
    \frac{\hat{r}_N\big(\hat{\mathcal T}_x^F\big)}{\hat{r}_N\big(\hat{\mathcal T}^F\big)} \,\approx\, \frac{\mu_M^N(x)\; \bar r_N\!\left(\bar{\mathcal T}_{M}\right)}{\sum_{M''\in\mathcal M}\bar r_N\!\left(\bar{\mathcal T}_{M''}\right)} \,=\,
    \mu_M^N(x)\,\; \bar\pi_N(M)
\end{equation}
using the Markov chain tree theorem for $\bar \pi_N$ for the latter equality.
Finally, taking the limits on both sides (that exist by Assumption \ref{FRL2} and Propositions \ref{existence}, \ref{lollodorme}, \ref{iodormo}) and using equation \eqref{duue}, we obtain $\pi(x) \,=\, \mu_M(x)\; \bar\pi(M)\,$. $\qed$
\smallskip

\section{Examples} \label{sec:examples}

We discuss simple examples of application of Theorems \ref{ilteo}, \ref{ilteo2}. In some cases we also discuss explicit computation of the arborescence weights.

\subsection{The case with $1$ scale}
Let $(V,E,r_N)$ be strongly connected graph such that the DAG constituted by the equivalence classes of the fast subgraph $(V,E_F)$ has a unique minimal class $M\in\mathcal M$. 
The iteration of Section \ref{sec: iteration} stops at level $k=1$.
\smallskip

By Proposition \ref{Unicomassimo} there exists an arborescence with all fast edges. Arborescences of this type can only be rooted at vertices of $M$ and are the fast arborescences in $\mathcal T^F$. 
By Lemma \ref{Lemma2}, arborescences in $\mathcal T^F$ restricted to $\left(M, E_F[M]\right)$ are precisely the arborescences of this subgraph. Then:
\begin{equation}
	r_N\!\left(\mathcal T^F\right) \,=\, r_N\!\left(\mathcal T^F[M]\right)\, r_N\!\left(\mathcal F_M^F\right) \,,
\end{equation}
where $\mathcal F_M^F$ is the set of fast spanning forests with roots in $M$.
We compute the limiting invariant measure $\pi$ using the Markov chain tree theorem considering only fast arborescences (Proposition \ref{existence}).
For $x\in V\setminus M$ Lemma \ref{Lemma1} gives $\lim_{N\to\infty} \pi_N(x)=0$, while for $x\in M$ we obtain
\begin{equation}
	\pi_N(x) \,\approx\, \frac{r_N\!\left(\mathcal T_x^F\right)}{r_N\!\left(\mathcal T^F\right)} \,=\, \frac{r_N\!\left(\mathcal T_x^F[M]\right)\,r_N\!\left(\mathcal F_M^F\right)}{r_N\!\left(\mathcal T^F[M]\right)\,r_N\!\left(\mathcal F_M^F\right)} \,=\,\frac{r_N\!\left(\mathcal T_x^F[M]\right)}{r_N\!\left(\mathcal T^F[M]\right)} \,=\, \mu_{M}^N(x)
\end{equation}
hence $\pi(x)=\mu_M(x)$ by taking $N\to\infty$.
\smallskip

\subsection{A simple example with $2$ scales}\label{k2}
A simple example where instead $k=2$ is the following.
We consider the Markov chain associated with the graph $(V,E)$ in Figure \ref{exxample}. The vertex set is $V=\{x,y,z\}$ and the edge set has two velocity classes. More precisely, the rates are defined as follows: 
\begin{equation}
r_N(z,x) =c(z,x)\,,\ \ 
r_N(z,y) = c(z,y)\,,\ \ 
r_N(y,x) = \frac{c(y,x)}{N}\,,\ \ 
r_N(x,z) = \frac{c(x,z)}{N}
\end{equation}
for a function $c: E\rightarrow (0,\infty)\,$ independent of $N$,
so that the fast rates are $O(1)$, while the slow ones are $O(1/N)$. 
Notice that $(V,E)$ is strongly connected but there are no arborescences composed only by fast edges.
The graph is so simple that the Markov chain tree theorem can be applied with explicit computations.
We have
\begin{equation}
\left\{
\begin{array}{l}
r_N\left(\mathcal T_x\right) \,=\, \frac{1}{N}\, c(y,x)\,\big(c(z,x)+c(z,y)\big) \\[2pt]
r_N\left(\mathcal T_y\right) \,=\, \frac{1}{N}\, c(x,z)\,c(z,y) \\[2pt]
r_N\left(\mathcal T_z\right) \,=\, \frac{1}{N^2}\, c(y,x)\,c(x,z)
\end{array}
\right.
\end{equation}
Then
\begin{equation} \label{uff2}
    \pi_N(x) \,=\, \frac{r_N(\mathcal{T}_x)}{r_N(\mathcal{T}_x)+r_N(\mathcal{T}_y)+r_N(\mathcal{T}_z)} \,\rightarrow\, \frac{c(y,x)\,\big(c(z,x)+c(z,y)\big)}{c(y,x)\,\big(c(z,x)+c(z,y)\big) +  c(x,z)\,c(z,y) } \;
\end{equation}
as $N\to\infty$.
Similar computations can be done for $\pi(y)$ and $\pi(z)$ obtaining
\begin{equation}\label{uff3}
 \pi(y) = \frac{c(x,z)\,c(z,y)}{c(y,x)\,\big(c(z,x)+c(z,y)\big) +  c(x,z)\,c(z,y) } \;,\quad \pi(z)=0 \;.
\end{equation}

\begin{figure}
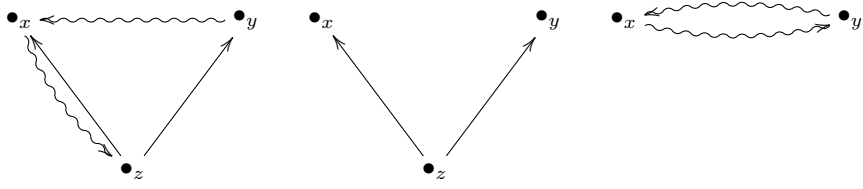

\centering
\xygraph{
!{<0cm,0cm>;<1cm,0cm>:<0cm,1cm>::}
!{(0,0) }*+{\bullet_{x}}="x"
!{(3,0) }*+{\bullet_{y}}="y"
!{(1.5,-2) }*+{\bullet_{z}}="z"
!{(4,0) }*+{\bullet_{x}}="a"
!{(7,0) }*+{\bullet_{y}}="b"
!{(5.5,-2) }*+{\bullet_{z}}="c"
!{(8,0) }*+{\bullet_{x}}="p"
!{(11,0) }*+{\bullet_{y}}="q"
"y":@{~>}"x"
"x":@/_/@{~>}"z"
"z":"x"
"z":"y"
"c":"a"
"c":"b"
"p":@/_/@{~>}"q"
"q":@/_/@{~>}"p"
}
\caption{From left to right: the strongly connected graph $(V,E)$, the fast subgraph $(V,E_F)$ coinciding with the DAG, and the reduced graph $(\hat V, \hat E)$ coinciding with the effective graph. Fast edges are represented by straight lines, slow edges by wavy lines.}
\label{exxample}
\end{figure}

From Figure \ref{exxample}, the equivalence classes of the fast graph $(V, E_F)$ consist of single nodes $\mathcal Q=\left\{\{x\},\{y\},\{z\}\right\}$ and the minimal classes are $\mathcal M=\left\{\{x\},\{y\}\right\}$. 
So $\hat V=\left\{x,y\right\}$ and the reduced rates can be computed directly:
\begin{equation}\label{uff}
\left\{
\begin{array}{l}
    \hat{r}_N(x,y) \,=\, r_N(x,z)\,P_N^{\{x,y\}}(y|z) \,=\, \frac{1}{N}\, c(x,z)\; \frac{c(z,y)}{c(z,x)+c(z,y)} \\[6pt]
    \hat{r}_N(y,x) \,=\, r_N(y,x) \,=\, \frac{1}{N}\,c(y,x)
\end{array}
\right.\;.
\end{equation}
Of course $\mu_{\{x\}}^N$ and $\mu_{\{y\}}^N$ are delta measures concentrated at $x$ and $y$ respectively.
Then the effective rates coincide with the reduced ones, precisely $\bar{r}_N\left(\{x\},\{y\}\right)=\hat{r}_N(x,y)$ and $\bar{r}_N\left(\{y\},\{x\}\right)=\hat{r}_N(y,x)$.
Both effective rates have the same scale, hence the invariant measure $\bar \pi_N$ is independent of $N$ and can be computed from \eqref{uff}:
\begin{equation}
	\bar\pi(\{x\}) \,=\, \frac{\hat r_N(y,x)}{\hat r_N(x,y)+ \hat r_N(y,x)} \,=\, \frac{ c(y,x)\,\big(c(z,x)+c(z,y)\big)}{ c(y,x)\,\big(c(z,x)+c(z,y)\big)\,+\, c(x,z)\,c(z,y)}
\end{equation}
and $\bar\pi(\{y\})=1-\bar\pi(x)\,$.
Theorem \ref{ilteo} states $\pi \,=\, \bar \pi(\{x\})\,\delta_{\{x\}} + \pi(\{y\})\, \delta_{\{y\}}$, which coincides with the result obtained by the direct computation \eqref{uff2}-\eqref{uff3}.
\smallskip

Notice that since the effective rates all have the same scale, the graph $\left(\mathcal M, \bar{E}_F\right)$ coincides with $\left(\mathcal M, \bar{E}\right)$ (strongly connected) and at the next iteration step $\mathcal M^{(2)}$ contains a single element $M^{(2)}=\{\{x\}, \{y\}\}$.
The probability measure $\bar\pi$ coincides with $\mu_{M^{(2)}}^{(1)}$. See Figure \ref{bohboh} for the corresponding forest representation of the limiting invariant measure $\pi$.


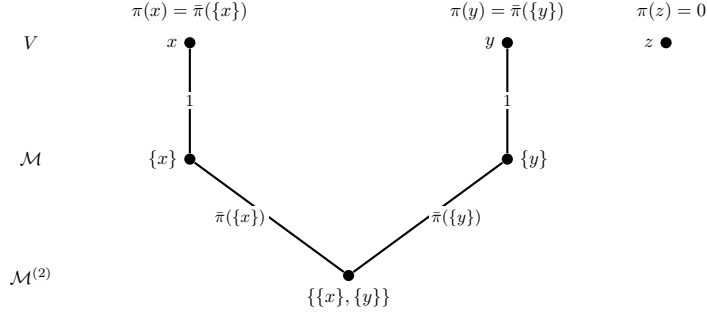
\begin{figure}
	\begin{tikzpicture}[
		scale=0.7, transform shape,
		level distance=2.2cm,
		dot/.style={circle, fill=black, minimum size=6pt, inner sep=0pt},
		edge from parent/.style={draw, thick},
		edge label/.style={midway, fill=white, inner sep=1.5pt, font=\small}
		]
		
		\begin{scope}[xshift=-11cm]
			\node at (0,0) {$\mathcal{M}^{(2)}$};
			\node at (0,2.2) {$\mathcal{M}$};
			\node at (0,4.4) {$V$};
			\node[dot, label=left:$z$] at (12,4.4) {};
			\node at (3,5) {$\pi(x)=\bar\pi(\{x\})$};
			\node at (9,5) {$\pi(y)=\bar\pi(\{y\})$};
			\node at (12.1,5) {$\pi(z)=0$};
		\end{scope}
		
		\node[dot , label=below:{$\{\{x\},\{y\}\}$}] (root) at (-5,0) {}
		child[grow=up, xshift=-3cm] {
			node[dot, label=left:$\{x\}$] {} 
			child {
				node[dot, label=left:$x$] {} 
				edge from parent node[edge label] {1}
			}
			edge from parent node[edge label, left] {$\bar\pi(\{x\})$}
		}
		child[grow=up, xshift=3cm] {
			node[dot, label=right:$\{y\}$] {}
			child {
				node[dot, label=left:$y$] {} 
				edge from parent node[edge label] {1}
			}
			edge from parent node[edge label, right] {$\bar\pi(\{y\})$}
		};
	\end{tikzpicture}
	\caption{Forest associated to the $2$-scale dynamics of Figure \ref{exxample}. 
	}
\label{bohboh}
\end{figure}
\smallskip

\subsection{A simple example with $k$ scales}
Consider a continuous-time random walk on the one-dimensional segment of length $k+1$, given by $V=\{0,1,\dots,k\}$, $E=\{(x,x+1)\}_{x=0,\dots,k-1}\,$. Choose the following symmetric transition rates:
\begin{equation}
	r_N(x,x+1) \,=\, r_N(x+1,x) \,:=\,  e^{-x\,N}\,,\quad x=0,\dots,k-1\,.
\end{equation}
\begin{figure}[h]
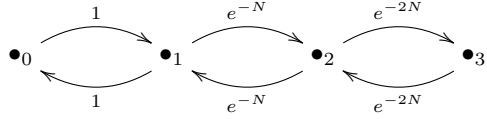

	\centering
	\mbox{ \xygraph{
			!{<0cm,0cm>;<1cm,0cm>:<0cm,1cm>::}
			!{(0,0) }*+{\bullet_{0}}="a"
			!{(2,0) }*+{\bullet_{1}}="b"
			!{(4,0) }*+{\bullet_{2}}="c"
			!{(6,0) }*+{\bullet_{3}}="d"
			"a":@/^0.4cm/^{1}"b"
			"b":@/^0.4cm/^{1}"a"
			"b":@/^0.4cm/^{{e^{-N}}}"c"
			"c":@/^0.4cm/^{{e^{-N}}}"b"
			"c":@/^0.4cm/^{{e^{-2N}}}"d"
			"d":@/^0.4cm/^{{e^{-2N}}}"c"
		}
	}
	\caption{Random walk on a one-dimentional lattice, $k=3$.}
		\label{trota}
\end{figure}
The walk goes more likely to the left, but the more it goes to the right the slower it becomes. Since the rates are symmetric, every rooted arborescence $\tau\in\mathcal T$ has the same weight $e^{-\frac{k(k-1)}{2}\,N}$, hence the invariant measure is independent of $N$ and uniform: $\pi_N(x)=\frac{1}{k+1}$ for all $x=0,\dots, k$.
\smallskip

At the first step of iteration described in Sections \ref{sec: main}-\ref{sec: iteration}, the only fast edges are $(0,1),\,(1,0)$ and the equivalence classes with respect to fast communication are $\mathcal Q^{(1)}=\mathcal M^{(1)}=\{\{0,1\},\{2\},\dots,\{k\}\}$. All of them are local minima of the DAG. In particular, we call $M^{(1)}:=\{0,1\}$ the left-most class. By symmetry, the invariant measure of the fast chain restricted to $M^{(1)}$ is independent of $N$ and given by $\mu_{M^{(1)}}^{(1)}(0)=\mu_{M^{(1)}}^{(1)}(1)=\frac{1}{2}\,$.
The effective rates are:
\begin{equation}\begin{cases}
\;\bar r_N^{(1)}\big(M^{(1)},\,\{2\}\big) \,=\, \mu_{M^{(1)}}^{(1)}(1)\;r_N(1,2) \,=\,\frac{1}{2}\,e^{-N}\\[4pt]
\;\bar r_N^{(1)}\big(\{2\},\,M^{(1)}\big) \,=\, \mu_{\{2\}}^{(1)}(2)\;r_N(2,1) \,=\, e^{-N}
\end{cases} \,. \end{equation}
%
At every iteration step $2\leq i\leq k$, the fast edges are the left-most ones 
and the equivalence classes with respect to fast communication are $\mathcal Q^{(i)}=\mathcal M^{(i)}=\{\{M^{(i-1)},\{i\}\},\{i+1\},\dots,\{k\}\}$. All of them are local minima of the DAG. In particular we call $M^{(i)}:=\{M^{(i-1)},\{i\}\}$ the left-most class, with support $\{0,1,\dots,i\}$. The invariant measure of the chain restricted to $M^{(i)}$ is given by
\begin{equation}\label{recmeas}
	\mu_{M^{(i)}}^{(i)}\big(\{i\}\big) \,=\, \frac{\bar r_N^{(i-1)}\big(M^{(i-1)},\{i\}\big)}{\bar r_N^{(i-1)}\big(M^{(i-1)},\{i\}\big)+\bar r_N^{(i-1)}\big(\{i\},M^{(i-1)}\big)}
\end{equation}
and $\mu_{M^{(i)}}^{(i)}\big(M^{(i-1)}\big)\,=\, 1- \mu_{M^{(i)}}^{(i)}\big(\{i\}\big)\,$; then the new effective rates are:
\begin{equation}\label{recrates} \begin{cases}
	\;\bar r_N^{(i)}\big(M^{(i)},\,\{i+1\}\big) \,=\, \mu_{M^{(i)}}^{(i)}\big(\{i\}\big)\;r_N(i,i+1) \,=\, \mu_{M^{(i)}}^{(i)}\big(\{i\}\big)\,e^{-i\,N}\\[4pt]
	\;\bar r_N^{(i)}\big(\{i+1\},\,M^{(i)}\big) \,= \,\mu_{\{i+1\}}^{(i)}(i+1)\;r_N(i+1,i) \,=\, e^{-i\,N}
\end{cases} \end{equation}
(if $i\leq k-1$). 
Applying relations \eqref{recmeas}, \eqref{recrates} recursively we find
\begin{equation}
	\mu_{M^{(i)}}^{(i)}\big(\{i\}\big) \,=\, \frac{1}{i+1}\,,\quad \mu_{M^{(i)}}^{(i)}\big(M^{(i-1)}\big) \,=\, \frac{i}{i+1}\,,\quad i=1,\dots,k\,.
\end{equation}
Therefore Theorem \ref{ilteo2} gives the following formula, represented in Figure \ref{forestrota}, for the limiting invariant measure of the whole chain:
\begin{equation}
	\pi(x) \,=\, \Big(\prod_{j=0}^{x-1}\mu_{\{j\}}^{(j)}(j)\Big)\;\  \mu_{M^{(x)}}^{(x)}\big(\{x\}\big)\;\ \Big(\prod_{i=x+1}^k \mu_{M^{(i)}}^{(i)}\big(M^{(i-1)}\big)\Big) \,=\, \frac{1}{k+1}
\end{equation}
(since the first $x-1$ terms in the product are $1$ and the other terms simplify), which is coherent with the explicit computation.
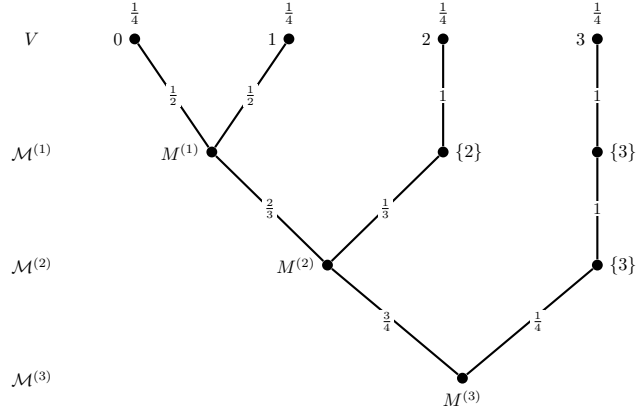
\begin{figure}
\centering
\begin{tikzpicture}[
    scale=0.68,
    transform shape,
    dot/.style={circle,fill=black,minimum size=6pt,inner sep=0pt},
    edge/.style={draw,thick},
    edge label/.style={midway,fill=white,inner sep=1pt,font=\small}
]
\def\yA{0}
\def\yB{2.2}
\def\yC{4.4}
\def\yD{6.6}

\coordinate (v0) at (-3,\yD);
\coordinate (v1) at (0,\yD);
\coordinate (v2) at ( 3,\yD);
\coordinate (v3) at ( 6,\yD);
\coordinate (m1) at (-1.5,\yC);      
\coordinate (s2) at ( 3,\yC);
\coordinate (s3a) at ( 6,\yC);
\coordinate (m2) at (0.75,\yB);       
\coordinate (s3b) at (6,\yB);
\coordinate (root) at (3.375,\yA);     

\node at (-5,\yA) {$\mathcal M^{(3)}$};
\node at (-5,\yB) {$\mathcal M^{(2)}$};
\node at (-5,\yC) {$\mathcal M^{(1)}$};
\node at (-5,\yD) {$V$};

\node at (-3,\yD+0.5) {$\frac14$};
\node at (0,\yD+0.5) {$\frac14$};
\node at ( 3,\yD+0.5) {$\frac14$};
\node at ( 6,\yD+0.5) {$\frac14$};

\node[dot,label=below:$M^{(3)}$] (R) at (root) {};

\node[dot,label=left:$M^{(2)}$] (M2) at (m2) {};
\node[dot,label=right:$\{3\}$] (S3b) at (s3b) {};

\node[dot,label=left:$M^{(1)}$] (M1) at (m1) {};
\node[dot,label=right:$\{2\}$] (S2) at (s2) {};
\node[dot,label=right:$\{3\}$] (S3a) at (s3a) {};

\node[dot,label=left:$0$] (V0) at (v0) {};
\node[dot,label=left:$1$] (V1) at (v1) {};
\node[dot,label=left:$2$] (V2) at (v2) {};
\node[dot,label=left:$3$] (V3) at (v3) {};

\draw[edge] (R)--(M2)
    node[edge label,left]{$\frac34$};

\draw[edge] (R)--(S3b)
    node[edge label,right]{$\frac14$};

\draw[edge] (M2)--(M1)
    node[edge label]{$\frac23$};

\draw[edge] (M2)--(S2)
    node[edge label]{$\frac13$};

\draw[edge] (M1)--(V0)
    node[edge label]{$\frac12$};

\draw[edge] (M1)--(V1)
    node[edge label]{$\frac12$};

\draw[edge] (S2)--(V2)
    node[edge label]{$1$};

\draw[edge] (S3b)--(S3a)
    node[edge label]{$1$};

\draw[edge] (S3a)--(V3)
    node[edge label]{$1$};

\end{tikzpicture}
\caption{Forest associated to the $3$-scale dynamics of Figure \ref{trota}.}
\label{forestrota}
\end{figure}
\smallskip

\subsection{Product state space with 2 scales}
We discuss a class of models that, in special cases, can be imagined as discrete counterparts of coupled diffusions evolving on different time scales as for example in \cite{EP,AN,CCKM,AMM}. 
\smallskip

Consider a finite vertex set of the form $V= X\times Y\,$: the possible states are the points $(x,y)\in V$.
Given a parameter $N>0$, transition rates are defined as follows
\begin{equation}\label{diffr}
	\,r_N\big((x,y),\,(x',y')\big) := \begin{cases}
	\;r^y(x,x')  & \textrm{if }x\neq x'\textrm{ and }y=y' \\[4pt]
	\;\frac{1}{N}\,r_x(y,y') & \textrm{if }x=x'\textrm{ and } y\neq y'
\end{cases}
\end{equation}
where $r^y:X\times X\to[0,\infty)$, $r_x:Y\times Y\to[0,\infty)$ are given and independent of $N$.
Coordinates $x,y$ change one at a time, for large $N$ the transitions in the $x$-direction are fast and the transitions in the $y$-direction are slow.
The edge set is
\begin{equation}
E:=\left\{\left((x,y),(x',y')\right)\in V\times V\,:\, \left(y=y',\,(x,x')\in E^y\right) \textrm{ or } \left(x=x',\,(y,y')\in E_x\right)  \right\} 
\end{equation}
where $E^y:=\{(x,x')\in X\times X \,:\, r^y(x,x')>0\}$ and $E_x:=\{(y,y')\in Y\times Y \,:\, r_x(y,y')>0\}$.
We assume that each graph $(X,E^y)$, $(Y,E_x)$ is strongly connected and we call $\mu_y$, $\nu_x$ the respective invariant measures (that do not depend on $N$). In particular the full graph $(V,E)$ is strongly connected and the continuous-time Markov chain with transition rates $r_N$ has a unique invariant measure $\pi_N$.
\smallskip

Fast communication partitions the vertex set in the equivalence classes $\mathcal M=\{M_y\}_{y\in Y}$, where $M_y:=X\times\{y\}$. They are all local minima of the DAG.
Since their support coincide with the whole vertex set $V$, the reduced rates $\hat r_N$ coincide with $r_N$. 
The effective rates are
\begin{equation} \label{eq:weff}
	\bar r_N\left(M_y,M_{y'}\right) = \frac 1N\sum_{x\in X} \,\mu_y(x)\; r_x(y,y') \;,\quad y,y'\in Y
\end{equation}
and since they all have the same scale, we renormalize them multiplying by $N$:
\begin{equation} \label{eq:weff2}
	\bar r\left(M_y,M_{y'}\right) := \sum_{x\in X} \,\mu_y(x)\; r_x(y,y')\,.
\end{equation}
The associated invariant measure $\bar\pi$ remains the same, independent of $N$.
No more iterations are needed (at level $k=2$, $\mathcal M^{(2)}$ has a single element) and by Theorem \ref{ilteo} we obtain
\begin{equation} \label{eq:limiteXY}
 \lim_{N\to\infty} \pi_N(x,y) \,=\, \bar\pi\left(M_y\right)\, \mu_y(x)\;, \quad (x,y)\in X\times Y \,.
\end{equation}
%
Now, assume that fixing $y\in Y$ or $x\in X$ the transition rates are reversible:
\begin{equation} \label{eq:xdetail}
	\mu_y(x)\ r^y(x,x') \,=\, \mu_y(x')\  r^y(x',x)\;,\quad x,x'\in X \,;
\end{equation}
\begin{equation} \label{eq:ydetail}
	\nu_x(y)\ r_x(y,y') \,=\, \nu_x(y')\  r_x(y',y)\;, \quad y,y'\in Y\,.
\end{equation}
Note that the global dynamics \eqref{diffr} and the effective dynamics \eqref{eq:weff} are in general not reversible.
To compare with the nice continuous framework \cite{ACG}, we look for conditions to guarantee \textit{reversibility of the effective dynamics}: this is the case in which the effective invariant measure $\bar\pi$ can be usually computed.

\begin{proposition} \label{prop:davide}
Assume reversibility $\eqref{eq:ydetail}$ of the rates $r_x$.
Suppose there exist two functions $a:X\to(0,\infty)$, $b:Y\to(0,\infty)$ such that
\begin{equation} \label{eq:alphabeta}
	a(x)\; \nu_x(y) \;=\; b(y)\; \mu_y(x)\,,\quad (x,y)\in V \,.
\end{equation}
Then the effective dynamics \eqref{eq:weff2} is reversible with respect to its unique invariant measure $\bar\pi$, given by
\begin{equation} \label{eq:pibeta}
\bar\pi\left(M_y\right) \,=\, \frac{1}{Z}\,b(y)\,,\quad y\in Y 
\end{equation}
where $Z=\sum_{y'\in Y}b(y')$.
\end{proposition}

\begin{proof}
We show that the effective rates satisfy the detailed balance equation. For every $y,y'\in Y$ using \eqref{eq:weff2}, \eqref{eq:pibeta} and \eqref{eq:alphabeta} we have:
\begin{equation}
	\bar\pi \left(M_y\right)\, \bar r\left(M_y,M_{y'}\right)  \,=\,
	\frac{b(y)}{Z}\, \sum_{x\in X} \,\mu_y(x)\; r_x(y,y') \,=\, \frac 1Z\,\sum_{x\in X}\,a(x)\;\nu_x(y)\; r_x(y,y') \,.
\end{equation}
Then by \eqref{eq:ydetail} we conclude that 
	$\bar\pi \left(M_y\right)\, \bar r\left(M_y,M_{y'}\right) \,=\, \bar\pi \left(M_{y'}\right)\; \bar r\left(M_{y'},M_{y}\right)\,$. 
\end{proof}

\begin{corollary}\label{cor:yprocess}
Let $\left(\mu_y\right)_{y\in Y}$ be an arbitrary collection of probability measures on $X$ and $\bar\mu$ an arbitrary probability measure on $Y$.
Then there exist rates $r_N$ as in \eqref{diffr} such that the limiting invariant measure is the prescribed one:
\begin{equation} \label{eq:prescribedlim}
	\lim_{N\to\infty}\pi_N(x,y) \,=\, \bar\mu(y)\; \mu_y(x)\;,\quad (x,y)\in X\times Y\,.
\end{equation}
More precisely:
\begin{itemize}
\item[i)]  reversible rates $r^y(x,x')$ with respect to the prescribed measure $\mu_y(x)$ can be chosen;\\ 
\item[ii)] 
reversible rates $r_x(y,y')$ with respect to $\nu_x(y)$ can be chosen with respect to the measure $\nu_x(y)$ defined as
\begin{equation}\label{eq:pi2}
\nu_x(y) := \frac{1}{Z_x}\;\mu_y(x)\; \bar\mu(y)\;,\quad y\in Y
\end{equation}
where $Z_x:=\sum_{y'\in Y} \mu_{y'}(x)\,\bar\mu(y')$; \\ 
\item[iii)] the effective rates $\bar r$ defined by \eqref{eq:weff2} are reversible with respect to the prescribed measure $\bar \mu(y)$. 
As a consequence, the effective invariant measure
\begin{equation} 
\bar\pi\left(M_y\right) \,=\, \bar\mu(y)\;,\quad y\in Y\,.
\end{equation}
\end{itemize}
\end{corollary}

\begin{proof}
The measures $\mu_y$ and $\bar\mu$ are given, then define $\nu_x$ as \eqref{eq:pi2}.
\smallskip
 
i), ii) are straightforward, since given a probability measure it is always possible to construct a reversible Markov chain with respect to that measure, for instance by employing Metropolis-Hastings dynamics.
\smallskip

To prove iii), set
\begin{equation}
 a(x) := Z_x = \sum_{y\in Y}\mu_y(x)\,\bar\mu(y) \;, \qquad
 b(y) := \bar\mu(y)\,,
\end{equation}
and notice that condition \eqref{eq:alphabeta} is verified.
Therefore by Proposition \ref{prop:davide} the effective rates $\bar r\left(M_y,M_{y'}\right):=\sum_{x\in X}\mu_y(x)\,r_x(y,y')$ are reversible with respect to the invariant measure $\bar\pi\left(M_y\right)=\bar \mu(y)\,$.
\smallskip

Finally, limit \eqref{eq:prescribedlim} follows by equation \eqref{eq:limiteXY}.
\end{proof}

We want to reproduce the discrete counterpart of the two-scale diffusion studied in \cite{AMM} and references therein, where the variables $x$ and $y$ are coupled by a potential but evolve on two different time scales with different \textit{effective temperatures}.
The concept of effective temperatures for each time scale arises in non-equilibrium systems such as turbulent flows \cite{HS} and spin glasses \cite{CKP}, appearing naturally as ratios of correlation and response functions, generalizing the usual fluctuation-dissipation relation (see for example \cite{BCKM,CR,C,GPSVV,M} for comprehensive reviews).
A coupling potential $U:X\times Y\to\R$ and $\beta_1,\beta_2>0$ are given.
On $X$ we consider the \textit{Gibbs measure} with potential $U(x,y)$ for fixed $y$ and inverse temperature $\beta_1$
\begin{equation} \label{eq:Gibbs}
	\mu_y(x) := \frac{1}{Z^y}\;e^{-\beta_1\,U(x,y)} 
\end{equation}
where $Z^{y}:=\sum_{x\in X}e^{-\beta_1\,U(x,y)}\,$. On $Y$ we consider the Gibbs measure with \textit{effective potential} $F(y):=-\frac{1}{\beta_1}\log Z^y$ and inverse temperature $\beta_2$
\begin{equation} \label{eq:effGibbs}
	\bar\mu(y) := \frac{1}{\bar Z}\;e^{-\beta_2\,F(y)}
\end{equation}
where $\bar Z:=\sum_{y\in Y}e^{-\beta_2\,F(y)}\,$.
Corollary \ref{cor:yprocess} guarantees that the prescribed joint measure is the limiting invariant measure of a suitable two-scale continuous-time Markov chain $(X_t^N,Y_t^N)$ on the state space $X\times Y$:
\begin{equation}
\lim_{N\to\infty} \pi_N(x,y) \,=\, \mu_y(x)\,\bar\mu(y) \;,\quad (x,y)\in X\times Y\,.
\end{equation}
Precisely, one has to choose on $X\times X$ reversible rates $r^y$ w.r.t. the Gibbs measure $\mu_y$, and on $Y\times Y$ reversible rates $r_x$ w.r.t. the measure $\nu_x$ provided by \eqref{eq:pi2}:
\begin{equation} \label{eq:nuXY}	
	\nu_x(y) = \frac{1}{Z_x}\,e^{-\beta_1\,U(x,y)\,-\,(\beta_2-\beta_1)\,F(y)} 
\end{equation}
where $Z_x:=\sum_{y'\in Y}e^{-\beta_1\,U(x,y')\,-\,(\beta_2-\beta_1)\,F(y')}\,$.
Notice that in particular cases the above measure  simplifies. Indeed, if the two temperatures coincide, i.e. $\beta_1=\beta_2=:\beta$, then $\nu_x(y) = \frac{1}{Z_x}\,e^{-\beta\,U(x,y)}$. If instead the potential is separable, i.e. $U(x,y)=U_1(x)+U_2(y)$, then the effective potential rewrites as $F(y)= U_2(y)+c$ and we get $\nu_x(y) = \frac{1}{Z_x}\,e^{-\beta_2\,U_2(y)}$.
Moreover, in the two-scale diffusion studied in \cite{AMM} a nice simplification is always possible, replacing the measure \eqref{eq:nuXY} by 
\begin{equation} \label{eq:nuXYsimple}
	\widetilde\nu_x(y) := \frac{1}{\widetilde Z_x}\, e^{-\beta_2\,U(x,y)} \;.
\end{equation}
The latter simplification is peculiar of continuous space framework: in general $\widetilde\nu_x$ does not satisfy condition \eqref{eq:alphabeta} of Proposition \ref{prop:davide}.
Nevertheless it might be possible to recover $\widetilde\nu_x$ when the graph $(Y, E_x)$ is a discrete grid approximating the torus $\T^{n}$ in the limit of the mesh size $h\to 0$. 
We only give the following  

\begin{proposition}
Let $U:X\times\T^{n}\to\R$ be a potential of regularity class $C^2$ with respect to $y\in\T^{n}$. $X$ finite set. $\beta_1,\beta_2>0$.
Let $\frac{1}{h}\in\N$ and $Y_h:=\{0,h,2h,\dots,1-h\}^{n}$ a discrete set of points in $\T^{n}$ with uniform spacing. Define the edge set $E_h:=\{(y,y')\in Y_h\times Y_h\,|\,\exists\,i:\, y_i'=y_i\pm h,\ y_j'=y_j\,\forall\,j\neq i \}$ with rates
\begin{equation} \label{eq:discreterates}
	r_x(y,y') \,:=\, e^{-\frac{1}{2}\beta_2\,(U(x,y')-U(x,y))} 
	\,,\quad  (y,y')\in E_h \,.
\end{equation}
Then:
\begin{itemize}
\item[i)] the rates $r_x(y,y')$ are reversible with respect to the measure $\widetilde\nu_x(y)$ defined by \eqref{eq:nuXYsimple};\\
\item[ii)] the effective rates $\bar r(y,y')$ defined by \eqref{eq:weff2},\eqref{eq:Gibbs} are approximately reversible with respect to the effective Gibbs measure $\bar\mu(y)$ defined by \eqref{eq:effGibbs}, precisely:
\begin{equation}
 \bar\mu(y)\; \bar r(y,y') - \bar\mu(y')\; \bar r(y',y) \,=\, O(h^2)\;\bar\mu(y)\,\bar r(y',y) \,.
\end{equation}
\end{itemize}
%
\end{proposition}

\begin{proof}
i) Assume $(y,y')\in E_h$, then by \eqref{eq:nuXYsimple} and \eqref{eq:discreterates}  both ratios $\frac{r_x(y,y')}{r_x(y',y)}$ and $\frac{\widetilde\nu_x(y')}{\widetilde\nu_x(y)}$ are equal to $e^{-\beta_2\,(U(x,y')-U(x,y))}$.
\smallskip

ii) Assume $(y,y')\in E_h$ with $y_i'=y_i\pm h$. We have:
\begin{equation}
	r_x(y,y') \,=\, 1-\frac{\beta_2}{2}\,\big(U(x,y')-U(x,y)\big) + O(h^2)\,=\, 1\mp \frac{\beta_2\,h}{2}\; \partial_{y_i} U(x,y) + O(h^2) 
\end{equation}
and a direct computation shows that
\begin{equation}
	\sum_{x\in X}  \mu_y(x)\; \partial_{y_i} U(x,y) \,=\, \partial_{y_i} F(y)
\end{equation}
where $\mu_y(x)$ is the Gibbs measure \eqref{eq:Gibbs} and $F(y)$ is the effective potential defined above \eqref{eq:effGibbs}. The latter is the crucial property in the continuous diffusion framework.
As a consequence the effective rates \eqref{eq:weff2} rewrite as
\begin{equation} \label{eq:discreteeffrates}
	\bar r(y,y') \,=\, 
	1\mp \frac{\beta_2\,h}{2}\, \partial_{y_i} F(y) + O(h^2) \,=\, 1- \frac{\beta_2\,h}{2}\, \big(F(y')-F(y)\big) + O(h^2)
\end{equation}
Then from \eqref{eq:discreteeffrates} and \eqref{eq:effGibbs} we obtain
\begin{equation} \begin{split}
	\frac{\bar r(y,y')}{\bar r(y',y)} \,=\, 
	1 - \beta_2\, \left(F(y')-F(y)\right) \,+ O(h^2)
	\,=\, \frac{\bar\mu(y')}{\bar\mu(y)} \,+O(h^2) \,. \qedhere
\end{split} \end{equation}
%
\end{proof}

\begin{remark}
The choice \eqref{eq:discreterates} can be generalized to rates
$r_x(y,y')$ of the form $w(y,y')\, f\big(\beta_2\left(U(x,y')-U(x,y)\right) \big)\,$,
where $w$ is a positive symmetric function on $E_h$ uniformly bounded away both from $0$ and $\infty$ and independent of $x$, and $f\in C^2(\R^2)$ is a positive function such that $f(z)/f(-z) = e^{-z}$ (e.g., Glauber $f(z)=\frac{1}{1+e^z}$). 
%
\end{remark}
\smallskip

\subsection{Boundary driven exclusion process}
We consider a one-dimensional segment with $n$ sites, where each site can accommodate at most one particle. 
We denote by $\eta\in\{0,1\}^n$ a generic configuration of particles. For  $x=1,\dots,n$, $\eta(x)=1$ means that there is a particle at site $x$, while $\eta(x)=0$ means that the site $x$ is empty.

The dynamics is described as follows. Particles jump to a nearest-neighbor site at rate 1, provided that the site is empty, so that the exclusion rule (allowing at most one particle per site) is preserved. The system is in contact with boundary sources, which evolve slowly compared to the bulk dynamics. In particular, at left boundary, i.e. $x=1$, a particle is created at rate $\alpha/N$ when the site is empty, while at the right boundary, i.e. $x=n$, the particles are destroyed at rate $\beta/N$ when the site is occupied. Here $\alpha,\beta$ are some fixed positive constants. The rules of this dynamics are illustrated in Figure \ref{dinex}. This is usually called the boundary driven exclusion process in weak contact with external reservoirs.

\begin{figure}[h]
\begin{tikzpicture}[scale=0.92,
	hop/.style={->,thick},
	curved/.style={->,thick,bend left=35},
	particle/.style={circle,fill=black,inner sep=2pt}
	]
	
	\def\nsites{12}
	\def\dx{0.9}
	\def\tick{0.12}
	
	\draw[thick] (0,0) -- ({(\nsites-1)*\dx},0);
	
	\foreach \i in {0,...,11} {
		\draw[thick] (\i*\dx,-\tick) -- (\i*\dx,\tick);
	}
	
	\draw[curved]
	(-0.9,0.4) to node[above] {$\frac{\alpha}{N}$} (0,0.12);
	
	\draw[curved]
	({11*\dx},0.12) to node[above] {$\frac{\beta}{N}$}
	({11*\dx+0.9},0.4);
	
	\node[particle] at ({4*\dx},0) {};
	\node[particle] at ({5*\dx},0) {};
	
\draw[curved]
({4*\dx},0.15) to ({5*\dx},0.15);

\node at ({4.5*\dx},0.30) {\Large$\times$};

	\node[particle] at ({8*\dx},0) {};
	
	\draw[curved]
	({8*\dx},-0.15) to node[below] {$1$} ({7*\dx},-0.20);
	
	\draw[curved]
	({8*\dx},0.15) to node[above] {$1$} ({9*\dx},0.20);
	
	\node[particle] at ({11*\dx},0) {};
	
\end{tikzpicture}
\caption{The jump rates of a one-dimensional boundary driven exclusion process in weak contact with external reservoirs.}
\label{dinex}
\end{figure}
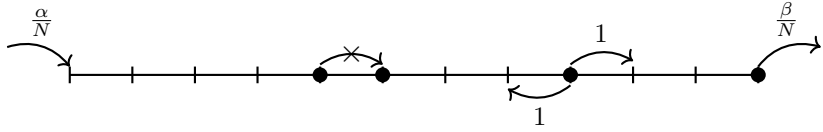

\smallskip
The transition graph is represented in Figure \ref{unparto}. Its vertices are the configurations $\eta\in V=\{0,1\}^n$.
The edges are of two types.
The fast edges have rate $r_N=1$ and correspond to particle jumps that preserve the total number of particles.
The slow edges have rates $r_N=\frac{\alpha}{N},\,\frac{\beta}{N}$ and correspond respectively to creation and annihilation of particles. 
The transition graph is strongly connected and has a positive invariant measure $\pi_N$.
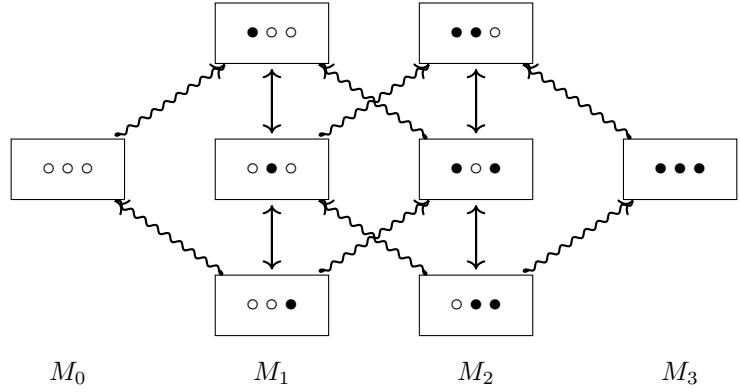
\begin{figure}
	\centering
	\begin{tikzpicture}[scale=0.9,
		vertex/.style={circle, fill=black, inner sep=1.5pt},
		every edge/.style={draw, thick, ->},
		wavy/.style={->, decorate, decoration={snake, amplitude=1pt, segment length=5.5pt}}]
		
		\node at (-6,-3) {$M_0$};
		\node at (-3,-3) {$M_1$};
		\node at (0,-3) {$M_2$};
		\node at (3,-3) {$M_3$};
		\node[draw, rectangle, minimum width=1.5cm, minimum height=0.8cm, align=center] (a) at (-6,0) {$\circ \circ \circ$};
		\node[draw, rectangle, minimum width=1.5cm, minimum height=0.8cm, align=center] (b) at (-3,2) {$\bullet \circ \circ$};
		\node[draw, rectangle, minimum width=1.5cm, minimum height=0.8cm, align=center] (c) at (-3,0) {$\circ \bullet \circ$};
		\node[draw, rectangle, minimum width=1.5cm, minimum height=0.8cm, align=center] (d) at (-3,-2) {$\circ \circ \bullet$};
		\node[draw, rectangle, minimum width=1.5cm, minimum height=0.8cm, align=center] (e) at (0,2) {$\bullet \bullet \circ$};
		\node[draw, rectangle, minimum width=1.5cm, minimum height=0.8cm, align=center] (f) at (0,0) {$\bullet \circ \bullet$};
		\node[draw, rectangle, minimum width=1.5cm, minimum height=0.8cm, align=center] (g) at (0,-2) {$\circ \bullet \bullet$};
		\node[draw, rectangle, minimum width=1.5cm, minimum height=0.8cm, align=center] (h) at (3,0) {$\bullet \bullet \bullet$};
		
		\draw[->, thick, shorten >=2pt, shorten <=2pt] (b) -- (c);
		\draw[->, thick, shorten >=2pt, shorten <=2pt] (c) -- (d);
		\draw[->, thick, shorten >=2pt, shorten <=2pt](c) -- (b);
		\draw[->, thick, shorten >=2pt, shorten <=2pt] (d) -- (c);
		\draw[->, thick, shorten >=2pt, shorten <=2pt] (e) -- (f);
		\draw[->, thick, shorten >=2pt, shorten <=2pt] (f) -- (e);
		\draw[->, thick, shorten >=2pt, shorten <=2pt](f) -- (g);
		\draw[->, thick, shorten >=2pt, shorten <=2pt] (g) -- (f);
		
		\draw[wavy, thick, shorten >=2pt, shorten <=3pt] (a) -- (b);
		\draw[wavy, thick, shorten >=2pt, shorten <=3pt] (d) -- (a);
		\draw[wavy, thick, shorten >=2pt, shorten <=3pt] (f) -- (b);
		\draw[wavy, thick, shorten >=2pt, shorten <=3pt] (c) -- (e);
		\draw[wavy, thick, shorten >=2pt, shorten <=3pt] (g) -- (c);
		\draw[wavy, thick, shorten >=2pt, shorten <=3pt] (d) -- (f);
		\draw[wavy, thick, shorten >=2pt, shorten <=3pt] (g) -- (h);
		\draw[wavy, thick, shorten >=2pt, shorten <=3pt] (h) -- (e);

\end{tikzpicture}
\caption{Transition graph associated to the boundary driven exclusion process on $n=3$ sites. Each vertex represents a configuration $\eta=(\eta_1,\eta_2,\eta_3)\in\{0,1\}^3$. Fast edges with rate $1$ are drawn as straight lines, slow edges with rate $\frac{\alpha}{N}$ or $\frac{\beta}{N}$ as wavy lines.
The number of particles in $\eta$ determines its equivalence class $M_i$: 
transitions inside the same class happen on time scale $O(1)$, transitions from a class to another happen on time scale $O(N)$.}
\label{unparto}
\end{figure}
\begin{figure} 
	\begin{tikzpicture}[
		scale=0.75, transform shape,
		level distance=2.2cm,
		dot/.style={circle, fill=black, minimum size=6pt, inner sep=0pt},
		edge from parent/.style={draw, thick},
		edge label/.style={midway, fill=white, inner sep=1.5pt, font=\small}
		]
		
		\begin{scope}[xshift=-10cm]
			\node at (-2,0) {};
			\node at (-2,2.2) {$|\mathcal{M}|=n+1$};
			\node at (-2,4.4) {$|V|=2^n$};
		\end{scope}
		
		\node[dot] (root) at (-4.8,0) {};
		
		\path (root)
		child[grow=up, xshift=-5.2cm] {
			node[dot, label=right:{$M_0$}] {}
			child {node[draw, rectangle, minimum width=1.5cm, minimum height=0.8cm, align=center] {$\circ \circ \circ$} }
		}
		child[grow=up, xshift=-2cm] {
			node[dot, label=right:{$M_1$}] {}
			child {node[draw, rectangle, minimum width=1.5cm, minimum height=0.8cm, align=center] {$\bullet \circ \circ$} }
			child {node[draw, rectangle, minimum width=1.5cm, minimum height=0.8cm, align=center] {$\circ \bullet \circ$} }
			child {node[draw, rectangle, minimum width=1.5cm, minimum height=0.8cm, align=center] {$\circ \circ \bullet$} }
		}
		child[grow=up, xshift=2.6cm] {
			node[dot, label=right:{$M_2$}] {}
			child {node[draw, rectangle, minimum width=1.5cm, minimum height=0.8cm, align=center] {$\bullet \bullet \circ$} }
			child {node[draw, rectangle, minimum width=1.5cm, minimum height=0.8cm, align=center] {$\bullet \circ \bullet$} }
			child {node[draw, rectangle, minimum width=1.5cm, minimum height=0.8cm, align=center] {$\circ \bullet \bullet$} }
		}
		child[grow=up, xshift=5.8cm] {
			node[dot, label=right:{$M_3$}] {}
			child {node[draw, rectangle, minimum width=1.5cm, minimum height=0.8cm, align=center] {$\bullet \bullet \bullet$} }
		};
\end{tikzpicture}
\caption{Forest associated to the process of Figure \ref{unparto}. 
}
\label{alberi-ex}
\end{figure}
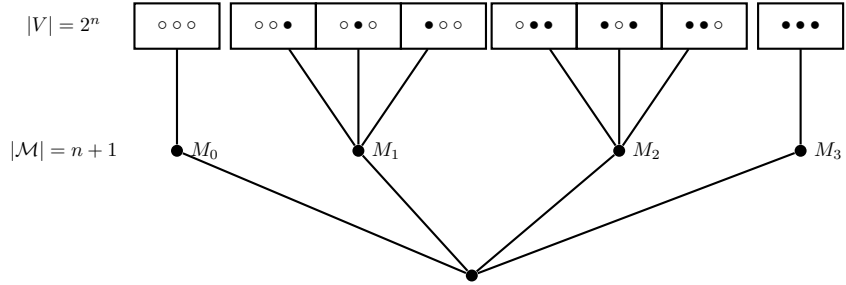
The multiscale iterative construction of Section \ref{sec: iteration} is encoded by a tree with $k=2$ levels, as illustrated in Figure \ref{alberi-ex}.

\smallskip
The equivalence classes of $(V,E_F)$ are sets of configurations with fixed number of particles: $M_i:=\left\{\eta\,:\, \sum_{x=1}^n\eta(x)=i\right\}$ for $i=0,\dots,n$. Each of them is a local minimum of the DAG. 
Since $\hat V=V$ the reduced rates coincide with original ones, i.e. $\hat r_N=r_N$.
\smallskip

The fast dynamics restricted to each equivalence class $M_i$ corresponds to the simple exclusion dynamics with closed boundaries.
Then the invariant measure $\mu_{M_i}^{N}$ does not depend on $N$ and is the uniform probability measure $\mu_i$ on $M_i$, i.e.,
\begin{equation}
	 \mu_i(\eta) \,=\, \frac{1}{{n\choose i}}\;\id\Big(\textstyle\sum_x\eta(x)=i\Big)\,.
\end{equation}
Under this measure the occupation probability $\mu_i\left(\eta(x)=1\right)$ is the same for every site $x$ and equals $\frac{i}{n}$.
\smallskip

The effective rates are $\bar{r}_N\left(M_i, M_j\right)=0$ if $|i-j|\neq 1$, and:
\begin{equation}\label{251}
	\bar{r}_N\!\left(M_i, M_{i+1}\right) \,=\, \dfrac{\alpha}{N}\; \mu_i\big(\eta(1)=0\big) \,=\, \dfrac{\alpha}{N}\,\Big(1-\dfrac{i}{n}\Big) \;,
\end{equation}
\begin{equation} \label{252}
	\bar{r}_N\!\left(M_i, M_{i-1}\right) \,=\, \dfrac{\beta}{N}\;\mu_i\big(\eta(n)=1\big) \,=\, \dfrac{\beta}{N}\,\dfrac{i}{n} \,.
\end{equation}
They all have the same scale so the invariant measure $\bar \pi_N$ does not depend on $N$.
The effective dynamics is a random walk on the integer numbers from $0$ to $n$ that increases or decreases its value by one with rates given in \eqref{251}-\eqref{252}. 
This Markov chain is reversible and $\bar\pi_N=\bar \pi$ can be computed by detailed balance equation, obtaining:
\begin{equation}
	\bar\pi\left(M_i\right) \,=\, \binom{n}{i}\left(\frac{\alpha}{\alpha+\beta}\right)^i\left(\frac{\beta}{\alpha+\beta}\right)^{n-i}\,,
\end{equation}
that is a Binomial distribution of parameters $n$ and $\frac{\alpha}{\alpha+\beta}$.
%
As a consequence, we can write the limiting invariant measure of the boundary driven exclusion process as
\begin{equation}
	\lim_{N\to\infty} \pi_N(\eta) \,=\, \sum_{i=0}^n\,\bar\pi\left(M_i\right)\,\mu_i(\eta) \,=\, \prod_{x=1}^n \Big(\frac{\alpha}{\alpha+\beta}\,\id\big(\eta(x)=1\big) \,+\, \frac{\beta}{\alpha+\beta}\,\id\big(\eta(x)=0\big) \Big)
\end{equation}
that is a Bernoulli product measure on $\{0,1\}^n$ of parameter $\frac{\alpha}{\alpha+\beta}$.
\smallskip

We can also obtain the above result by a direct identification of the fast arborescences. 
Fix a configuration with $i$ particles $\eta\in M_i$. The fast arborescences $\tau\in\mathcal T^F_\eta$ can be identified by using Lemma \ref{Lemma2}.
\smallskip

First of all, when restricted to a fast equivalence class, $\tau[M_j]$ must be an arborescence of $(M_j,E_F[M_j])$ and all its edge weights are equal to $1$.
$\tau$ contains no edge going out from the class of the root $M_i$.
\smallskip

In addition, for every $j>i$, $\tau$ contains exactly one edge going from class $M_j$ to class $M_{j-1}$, which has weight $\frac{\beta}{N}$.
For every $j<i$ instead, $\tau$ contains exactly one edge going from class $M_j$ to class $M_{j+1}$, which has weight $\frac{\alpha}{N}$.
These slow edges must start from the root of the sub-arborescence $\tau[M_j]$ they are leaving and the arrival point is automatically determined.
\smallskip

As a consequence, for $j>i$ the sub-arborescence $\tau[M_j]$ must be rooted at a configuration $\eta_j\in M_j$ such that $\eta_j(n)=1$, since with the next transition a particle has to be annihilated from the right boundary. The possible choices for such root $\eta_j$ are ${n-1 \choose j-1}$.

On the other hand, for $j<i$ the sub-arborescence $\tau[M_j]$ must be rooted at a configuration $\eta_j\in M_j$ such that $\eta_j(1)=0$, since with the next transition a particle has to be created from the left boundary. The possible choices for such root $\eta_j$ are ${n-1 \choose j}$.
\smallskip

Now observe that when configurations $\eta',\eta''$ both belong to the same class $M_j$ the rates $r_N(\eta',\eta'')$ are symmetric . Therefore to choose an arborescence $\tau[M_j]$ of $(M_j,E[M_j])$ rooted at $\eta_j$ it suffices to choose an un-directed spanning tree of this subgraph, and after that it is always possible to orient the edges of the tree towards $\eta_j$. We call $\mathcal N_j$ the number of spanning un-directed trees of $(M_j,E[M_j])$.
\smallskip

From the above considerations we deduce that for $\eta\in M_i$
\begin{equation}
	r_N\left(\mathcal T^F_\eta\right) \,=\, \frac{\alpha^i\,\beta^{n-i}}{N^n}\ \prod_{k=0}^{n-1}\binom{n-1}{k}\ \prod_{\ell=0}^n\mathcal N_\ell \;,
\end{equation}
where on the right hand side the first factor weights the slow edges connecting an equivalence class to another, the middle factors compute the possible choices of the local roots inside each class, and the last product accounts for the unidirected spanning trees of each class.
Notice that a part from the common factor that does not depend on the configuration $\eta$, we have found $r_N\left(\mathcal T^F_\eta\right) \,\propto\,\alpha^{\sum_x\eta(x)}\;\beta^{n-\sum_x \eta(x)}$ hence $\frac{r_N\left(\mathcal T^F_\eta\right)}{r_N\left(\mathcal T^F\right)}$ is a Bernoulli product measure of parameter $\frac{\alpha}{\alpha+\beta}$. 
By Proposition \ref{existence} this is the limiting invariant measure $\pi$.

\begin{section}*{Acknowledgments}
	 The authors thank Claudio Landim and Emanuele Mingione for useful discussions.
	 D.A. and D.G. acknowledge financial support from the Italian Ministry of University and Research 
	 through PRIN project “Emergence of condensation-like phenomena in interacting particle systems: kinetic and lattice models”, 
	 grant 202277WX43.
	 D.G. acknowledges financial support from the Italian Ministry of Foreign Affairs and International Cooperation by the grant BR26GR05. 
	 The authors are members of the Italian Institute of Advanced Mathematics (INdAM).
\end{section}

\end{document}